\documentclass[a4paper,12pt]{article} 

\usepackage{latexsym}
\usepackage{mathrsfs}
\usepackage{amsfonts,amsmath,amssymb}
\usepackage{indentfirst}
\usepackage{subeqnarray}
\usepackage{verbatim}
\usepackage[pdftex]{graphicx}
\usepackage{epstopdf}
\usepackage{bm}

\newcommand{\bx}{\mbox{\boldmath{$x$}}}

\newcommand{\btau}{\mbox{\boldmath{$\tau$}}}
\newcommand{\bzero}{\mbox{\boldmath{$0$}}}

\newcommand{\bI}{\mbox{\boldmath{$I$}}}

\newcommand{\fb}{\mbox{\boldmath{$f$}}}

\newcommand{\bu}{\mbox{\boldmath{$u$}}}

\newcommand{\bv}{\mbox{\boldmath{$v$}}}

\newcommand{\bw}{\mbox{\boldmath{$w$}}}

\newcommand{\bvarepsilon}{\mbox{\boldmath{$\varepsilon$}}}
\newcommand{\bsigma}{\mbox{\boldmath{$\sigma$}}}

\newcommand{\bnu}{\mbox{\boldmath{$\nu$}}}

\newcommand{\bxi}{\mbox{\boldmath{$\xi$}}}

\newtheorem{theorem}{Theorem}[section]
\newtheorem{lemma}[theorem]{Lemma}

\newtheorem{corollary}[theorem]{Corollary}

\newtheorem{problem}[theorem]{Problem}

\numberwithin{equation}{section}
\newenvironment{proof}[1][Proof]{\textbf{#1.} }
{\ \rule{0.75em}{0.75em}\smallskip}

\usepackage[pdftex,unicode,colorlinks,linkcolor=blue]{hyperref}

\begin{document}

\begin{center}
\Large\bf Numerical Studies of a Hemivariational Inequality for a Viscoelastic Contact Problem with Damage
\end{center}

\smallskip
\begin{center}
{\large Weimin Han}\footnote{Program in Applied Mathematical and Computational Sciences (AMCS) \&
Department of Mathematics, University of Iowa, Iowa City, IA 52242, USA.
Email: {\tt weimin-han@uiowa.edu}},\quad
{\large Michal Jureczka}\footnote{Jagiellonian University in Krakow, Faculty of Mathematics and Computer Science,
Lojasiewicza 6, 30-348 Krakow, Poland. Email: {\tt michal.jureczka@uj.edu.pl}}\quad {\large and} \quad
{\large Anna Ochal}\footnote{Jagiellonian University in Krakow, Faculty of Mathematics and Computer Science,
Lojasiewicza 6, 30-348 Krakow, Poland. Email: {\tt ochal@ii.uj.edu.pl}}
\end{center}

\medskip
\begin{quote}
{\bf Abstract.} \ This paper is devoted to the study of a hemivariational inequality modeling the 
quasistatic bilateral frictional contact between a viscoelastic body and a rigid foundation.  
The damage effect is built into the model through a parabolic differential inclusion for the 
damage function.  A solution existence and uniqueness result is presented.  A fully discrete 
scheme is introduced with the time derivative of the damage function approximated by the 
backward finite different and the spatial derivatives approximated by finite elements.  An optimal 
order error estimate is derived for the fully discrete scheme when linear elements are used for the 
velocity and displacement variables, and piecewise constants are used for the damage function.  
Simulation results on numerical examples are reported illustrating the performance of the 
fully discrete scheme and the theoretically predicted convergence orders.
\end{quote}



\smallskip

\section{Introduction}

In this paper, we study a mathematical model in the form of a hemivariational inequality for a 
quasistatic bilateral frictional contact problem between a viscoelastic body and a rigid 
foundation.  The friction law is given in the form of subdifferential condition.  Damage of the 
material is incorporated.  Modeling, variational analysis and numerical solution of contact problems
have been studied extensively; in this regard, a few comprehensive references are 
\cite{KO1988, HHNL1988, HS2002, SHS2006} in the context of variational inequalities (VIs), and
\cite{MOS2013, SM2018} in the context of hemivariational inequalities (HVIs).  

The notion of hemivariational inequalities (HVIs) was introduced in early 1980s to model mechanical 
problems involving non-smooth, non-monotone or multi-valued relations (\cite{Pa83}).  Early results on 
modeling, mathematical analysis and engineering applications of HVIs are summarized in \cite{Pa1993, NP1995};
recent summarized accounts include \cite{CLM2007, MOS2013, SM2018}.
Since there are no solution formulas for HVIs in applications, numerical simulation is the only 
feasible approach to solving HVIs. Detailed discussion of the finite element method for solving HVIs can be found in
\cite{HMP1999}.  More recently, there has been substantial progress in numerical analysis of HVIs,
especially on optimal order error estimates for numerical solutions of HVIs, starting with the 
paper \cite{HMS14}, followed by a sequence of papers, e.g., \cite{BBHJ15, HSB17, Ha18, HSD18}; 
the reader is referred to \cite{HS19AN} for a recent survey.  

Many contact processes are accompanied with material damage.  In applications, it is very important
to consider the damage effect.  General mathematical models for damage were derived in \cite{FN95, FN96};
see also \cite{Fr2002}.  In \cite{HSS01}, a quasistatic contact problem for a viscoelastic material 
is studied variationally and numerically, where the damage effect of the viscoelastic material
is taken into account.  Systematic variational analysis
and numerical analysis of contact problems with damage effect is summarized in \cite{SHS2006}.  
The mathematical problems investigated in these references are in the form of VIs.  For studies of 
contact problems with damage in the form of HVIs, the reader is referred to \cite{GOS15}.

This is the first paper devoted to numerical analysis of an HVI arising in a contact problem with damage.
The rest of the paper is organized as follows.  In Section \ref{sec:model}, we introduce the 
contact problem, present its weak formulation as an HVI and comment on the solution existence and uniqueness.
In Section \ref{sec:NA}, we consider a fully discrete numerical scheme for the contact problem
and derive an optimal order error estimate under appropriate solution regularity assumptions.
In Section \ref{sec:ex}, we report computer simulation results on a numerical example and illustrate 
numerical convergence orders that match the theoretical error bound.

\section{The contact problem}\label{sec:model}

We first introduce the pointwise formulation of the quasistatic contact problem for the contact 
between a viscoelastic body and a rigid foundation.  The initial configuration of the body is
$\Omega$, a Lipschitz bounded domain in $\mathbb{R}^d$ ($d\le 3$ in applications).  The body is 
subject to the action of volume forces of a total density $\fb_0$.  The boundary $\Gamma$
of the domain $\Omega$ is split into three disjoint measurable parts, $\Gamma=\Gamma_1\cup\Gamma_2\cup\Gamma_3$
such that $\Gamma_1$ is non-trivial.  We will assume the body is fixed along $\Gamma_1$, 
is subject to the action of surface tractions with a total density $\fb_2$.  Along the contact boundary 
$\Gamma_3$, the body and the foundation are in bilateral contact and the frictional process is
described by a generalized subdifferential inclusion.

Following \cite{CS1982,HS2002}, we consider a viscoelastic constitutive law with the damage effect in the form
\begin{equation*}
 \bsigma={\cal A}\bvarepsilon(\dot{\bu})+{\cal B}(\bvarepsilon(\bu),\zeta),
\end{equation*}
where $\bu$ is the displacement field, $\zeta$ is the damage function, $\bsigma$ is the stress field,
${\cal A}$ and ${\cal B}$ are the viscosity operator and the elasticity operator.  These operators
are allowed to depend on the spatial location.  For convenience, we use the shorthand notation
${\cal A}\bvarepsilon(\dot{\bu})$ and ${\cal B}(\bvarepsilon(\bu),\zeta)$ for ${\cal A}(\bx,\bvarepsilon(\dot{\bu}))$ and
${\cal B}(\bx,\bvarepsilon(\bu),\zeta)$, respectively.  The symbol $\dot{\bu}$ denotes the time derivative of $\bu$.
The time interval of interest is $[0,T]$ for some $T>0$.  

The notion of the damage function was introduced in \cite{FN95, FN96} to quantify the damage to the 
material.  It is defined to be the ratio between the elastic modulus of the damaged material and that
of the original material.  The value of the damage function $\zeta$ lies in $[0,1]$.  
The value $\zeta=1$ indicates that there is no damage in the material, whereas the value $\zeta=0$ 
corresponds to a completely damaged material.  When $0<\zeta<1$, there is a partial damage and the system 
has a reduced load carrying capacity.  A popular model for the evolution of the damage function is 
given by a parabolic differential inclusion:
\[ \dot \zeta - \kappa\, \triangle\zeta + \partial I_{[0,1]}(\zeta) \ni\phi(\bvarepsilon(\bu), \zeta), \]
where $\kappa>0$ is a constant microcrack diffusion coefficient, $I_{[0,1]}$ is the indicator function of 
the interval $[0,1]$, $\partial I_{[0,1]}$ is the convex subdifferential of $I_{[0,1]}$, and $\phi$ is
the mechanical source of damage, depending on the strain and the damage itself.  On the boundary $\Gamma$,
a homogeneous Neumann condition is described for $\zeta$.

For a vector $\bv$ defined on $\Gamma$, we let $v_\nu=\bv\cdot\bnu$ be its normal component, and 
let $\bv_\tau=\bv-v_\nu \bnu$ be its tangential component.  For a stress tensor $\bsigma$ defined on $\Gamma$, 
we let $\sigma_\nu=(\bsigma\bnu)\cdot\bnu$ and $\bsigma_\tau=\bsigma\bnu-\sigma_\nu \bnu$ be its normal 
and tangential components, respectively. 

Denote by $\bu_0$ and $\zeta_0$ the initial values of the displacement and the damage function.
The pointwise formulation of the contact problem is as follows.

\begin{problem}\label{P0}
Find a displacement field $\bu:\Omega\times[0,T]\to \mathbb{R}^d$, a stress field 
$\bsigma:\Omega\times[0,T]\to\mathbb{S}^d$, and a damage field $\zeta: \Omega\times[0,T] \to\mathbb{R}$ such that
\begin{align}
\bsigma & ={\cal A}\bvarepsilon(\dot{\bu}) + {\cal B}(\bvarepsilon(\bu), \zeta) & {\rm in}\ 
\Omega\times(0,T), \label{eq:11}\\[2pt]
\dot \zeta - \kappa\, \triangle\zeta + \partial I_{[0,1]}(\zeta) &\ni\phi(\bvarepsilon(\bu), \zeta)
&{\rm in}\ \Omega\times(0,T),\label{eq:12}\\[2pt]
{\rm Div}\,\bsigma + {\fb}_0& =\bzero &{\rm in}\ \Omega\times(0,T), \label{eq:13}\\[2pt]
\frac{\partial \zeta}{\partial\nu}& =0  &{\rm on}\ \Gamma\times(0,T), \label{eq:14}\\[2pt]
\bu& =\bzero  &{\rm on}\ \Gamma_1\times(0,T),\label{eq:15}\\[2pt]
\bsigma\bnu& ={\fb}_2  &{\rm on}\ \Gamma_2\times(0,T),\label{eq:16}\\
u_\nu& =0,\ -\bsigma_{\tau}\in \partial j(\dot\bu_{\tau}) & {\rm on}\ \Gamma_3\times (0,T),\label{eq:17}\\
\bu(0) & =  \bu_0,\ \zeta(0)=\zeta_0 &{\rm in}\ \Omega.\label{eq:18}
\end{align}
\end{problem}

We already know that \eqref{eq:11} is the viscoelastic constitutive law with damage, and \eqref{eq:12} is 
the evolution relation for the damage function.  We consider a quasistatic contact process and 
\eqref{eq:13} is the corresponding equilibrium equation.  The initial conditions for the displacement field 
and the damage function are given by \eqref{eq:18}.  The relations \eqref{eq:14}--\eqref{eq:17}
are the boundary condition for the damage function, the displacement boundary condition on $\Gamma_1$, the 
traction boundary condition on $\Gamma_2$, and the bilateral friction contact condition on $\Gamma_3$. 
Here, the friction dissipation pseudopotential $j$ will be assumed to be Lipschitz continuous, and 
$\partial j$ represents the generalized subdifferential in the sense of Clarke (cf.\ \cite{Cl75, Cl1983}).  
We will also need the notion of the generalized directional derivative in the sense of Clarke.
Let $V$ be a Banach space and let $\psi \colon V\to \mathbb{R}$ be a locally Lipschitz continuous functional. 
Recall that the generalized directional derivative of $\psi$ at $u\in V$ in the direction $v\in V$ is
\[ \psi^{0}(u;v):= \limsup_{w\to u,\,\lambda\downarrow 0}\frac{\psi(w +\lambda v)-\psi(w)}{\lambda}\,, \]
whereas the generalized subdifferential of $\psi$ at $u\in V$ is
\[ \partial\psi(u):=\left\{\xi\in V^*\mid\psi^0(u;v)\ge\langle\xi,v\rangle_{V^*\times V}\ \forall\,v\in V\right\}. \]
We note the following properties:
\begin{align}
& \psi^0(u;t\,v)=t\,\psi^0(u;v)\quad\forall\,u,v\in V,\,t\ge 0, \label{int5a}\\
& \psi^0(u;v_1+v_2)\le \psi^0(u; v_1)+\psi^0(u;v_2)\quad\forall\,u,v_1,v_2\in V, \label{int6}\\
& \psi^0(u;v) =\max\left\{ \langle\zeta, v\rangle_{V^*\times V}\mid\zeta\in\partial\psi(u)\right\}\quad\forall\,u,v\in V,
\label{int7}\\
& u_n\to u\ {\rm and}\ v_n\to v\ {\rm in}\ V\quad\Longrightarrow\quad
\limsup_{n\to\infty} \psi^0(u_n;v_n)\le \psi^0(u;v).
\label{int8}
\end{align}

Problem \ref{P0} will be studied in its weak form.  For this purpose, we first need to introduce some 
function spaces.  Let 
\[ Q=L^2(\Omega)^{d\times d}_{\rm sym}, \]
which is a Hilbert space with the inner product
\[ (\bsigma,\btau)_Q=\int_\Omega \sigma_{ij}\tau_{ij} dx,\quad \bsigma,\btau\in Q.\]
This will be the space for stress and strain fields.  The function space for the
displacement field is the Hilbert space 
\[ V=\left\{\bv\in H^1(\Omega)^d\mid \bv=\bzero\ {\rm on}\ \Gamma_1,\ v_\nu=0\ {\rm on}\ \Gamma_3\right\} \]
with the inner product $(\bu,\bv)_V= (\bvarepsilon(\bu), \bvarepsilon(\bv))_Q$ and the associated norm 
$\|\bv\|_V= \|\bvarepsilon(\bv)\|_Q$.   The space for the damage field is $Z=H^1(\Omega)$.  For convenience, 
we let $Z_0=L^2(\Omega)$.  The spaces $Z$ and $Z_0$ are endowed with their canonical inner products and norms.

In the study of the contact problem, we assume that the operator ${\cal A}\colon \Omega\times
\mathbb{S}^d\to \mathbb{S}^d$ satisfies the following conditions:
\begin{equation}
\left. \begin{array}{ll} 
{\rm (a)\  There\ exists}\ L_{\cal A}>0\ {\rm such\ that}\\
   {}\qquad \|{\cal A}(\bx,\bvarepsilon_1)-{\cal A}(\bx,\bvarepsilon_2)\|
      \le L_{\cal A} \|\bvarepsilon_1-\bvarepsilon_2\|
    \quad \forall\,\bvarepsilon_1,\bvarepsilon_2\in \mathbb{S}^d,\ {\rm a.e.}\ \bx\in \Omega.\\
  {\rm (b)\  There\ exists}\ m_{\cal A}>0\ {\rm such\ that}\\
    {}\qquad ({\cal A}(\bx,\bvarepsilon_1)-{\cal A}(\bx,\bvarepsilon_2))
       \cdot(\bvarepsilon_1-\bvarepsilon_2)\ge m_{\cal A}\,\|\bvarepsilon_1-\bvarepsilon_2\|^2\\
      {}\qquad\quad \forall\,\bvarepsilon_1,\bvarepsilon_2 \in \mathbb{S}^d,\ {\rm a.e.}\ \bx\in \Omega.\\
{\rm (c)\ For\ any\ }\bvarepsilon\in \mathbb{S}^d,\ \bx\mapsto
        {\cal A}(\bx,\bvarepsilon)\ {\rm is\ measurable\ on\ }\Omega.\\
{\rm (d)\ The\ mapping\ } \bx\mapsto {\cal A}(\bx,\bzero)\
{\rm belongs\ to}\ Q.
\end{array}\right\}
\label{eq:2}
\end{equation}
Similarly, we assume the operator ${\cal B}\colon \Omega\times\mathbb{S}^d\times\mathbb{R}\to \mathbb{S}^d$
has the following properties:
\begin{equation}
\left. \begin{array}{ll} {\rm (a)\ There\ exists\ }L_{\cal B}>0 {\rm\ such\ that}\\
  {} \qquad  \|{\cal B}(\bx,\bvarepsilon_1,\zeta_1)-{\cal B}(\bx,\bvarepsilon_2,\zeta_2)\|\le
    L_{\cal B}\,(\|\bvarepsilon_1-\bvarepsilon_2\|+|\zeta_1-\zeta_2|) \\
   {} \qquad \quad\forall\,\bvarepsilon_1,\bvarepsilon_2\in
     \mathbb{S}^d,\ \zeta_1,\zeta_2\in\mathbb{R},\ {\rm a.e.}\ \bx\in\Omega.\\
{\rm (b)\ For\ any\ }\bvarepsilon\in \mathbb{S}^d\ {\rm and}\ \zeta\in\mathbb{R},\ \bx\mapsto
   {\cal B}(\bx,\bvarepsilon,\zeta)\ {\rm is\ measurable\ on\ }\Omega.\\
{\rm (c)\ The\ mapping\ } \bx\mapsto {\cal B}(\bx,\bzero,0)\ {\rm belongs\ to}\ {Q}.
   \end{array}\right\}
\label{eq:3}
\end{equation}

As an example of the viscoelastic constitutive law with damage, we consider
\begin{equation}
\bsigma={\cal A}\bvarepsilon(\dot{\bu})+\eta\left(\bvarepsilon(\bu)-{\cal P}_{K(\zeta)}(\bvarepsilon(\bu))\right),
\label{vis1}
\end{equation}
where the viscosity tensor ${\cal A}$ satisfies \eqref{eq:2}, $\eta$ is a positive coefficient, $K(\zeta)$
is a damage dependent elasticity set, which is assumed to be convex and ${\cal P}_{K(\zeta)}$ is the projection operator onto the 
set $K(\zeta)$.  We require the properties $\bzero\in K(\zeta)$ and $\zeta_1\ge \zeta_2$ implies 
$K(\zeta_1)\subset K(\zeta_2)$.  The second property implies that as the damage of the material increases,
i.e., the value of the damage function $\zeta$ decreases, the elasticity convex set expands, and the 
material resembles a purely viscous one.   A concrete example is given by the von Mises convex set
\begin{equation}
K(\zeta)=\left\{\btau\in\mathbb{S}^d \mid \|\btau^D\|\le \zeta\,\sigma_Y\right\},
\label{vis2}
\end{equation}
where $\btau^D=\btau-({\rm tr}\,\btau/d)\,\bI$ is the deviatoric part of $\btau$, and $\sigma_Y>0$ is the 
yield limit of the damage-free material.  Since the projection operator is a contraction, it can be 
verified that ${\cal B}(\bvarepsilon,\zeta)=\eta\left(\bvarepsilon(\bu)-{\cal P}_{K(\zeta)}(\bvarepsilon(\bu))\right)$
satisfies \eqref{eq:3}.

On the damage source function $\phi\colon \Omega\times\mathbb{S}^d\times\mathbb{R}\to \mathbb{R}$,
the assumptions are
\begin{equation}                                           
\left. \begin{array}{ll} {\rm (a)\ There\ exists\ }L_\phi>0 {\rm\ such\ that}\\
  {} \qquad  |\phi(\bx,\bvarepsilon_1,\zeta_1)-\phi(\bx,\bvarepsilon_2,\zeta_2)|
    \le   L_\phi\,(\|\bvarepsilon_1-\bvarepsilon_2\|+|\zeta_1-\zeta_2|) \\
   {} \qquad\quad\forall\,\bvarepsilon_1,\bvarepsilon_2\in\mathbb{S}^d,\ \zeta_1,\zeta_2\in\mathbb{R},
   \ {\rm a.e.}\ \bx\in\Omega.\\
{\rm (b)\ For\ any\ }\bvarepsilon\in \mathbb{S}^d\ {\rm and}\ \zeta\in\mathbb{R},\ \bx\mapsto
   \phi(\bx,\bvarepsilon,\zeta)\ {\rm is\ measurable\ on\ }\Omega.\\
{\rm (c)\ The\ mapping\ } \bx\mapsto \phi(\bx,\bzero,0)\ {\rm belongs\ to}\ L^2(\Omega).
\end{array}\right\}
\label{eq:19}
\end{equation}
On the friction dissipation pseudopotential $j\colon \Gamma_3\times\mathbb{R}^d\to\mathbb{R}$,
\begin{equation}                                           
\left. \begin{array}{ll} {\rm (a)\ } \bx\mapsto j(\bx,\bxi)\ {\rm is\ measurable\ on}\ \Gamma_3\ 
     \forall\,\bxi\in\mathbb{R}^d\\
    {}\qquad {\rm and}\ j(\bx,\bzero)\in L^2(\Gamma_3).\\
    {\rm (b)\ } \bxi\mapsto j(\bx,\bxi)\ {\rm is\ locally\ Lipschitz\ in}\ \mathbb{R}^d,\, {\rm a.e.}\ \bx\in\Gamma_3.\\
    {\rm (c)\ } {\rm There\ exist\ constants}\ c_{0\tau},c_{1\tau}\ge 0\ {\rm such\ that}\\
    {}\qquad \|\partial j(\bx,\bxi)\|\le c_{0\tau}+c_{1\tau}\|\bxi\|\
     \forall\,\bxi\in \mathbb{R}^d,\,{\rm a.e.}\ \bx\in\Gamma_3.\\
    {\rm (d)\ } {\rm There\ is\ a\ constant}\ c_{2\tau}\ge 0\ {\rm such\ that}\\
    {}\qquad j^0(\bx,\bxi_1;\bxi_2-\bxi_1)+j^0(\bx,\bxi_2;\bxi_1-\bxi_2)
      \le c_{2\tau}\|\bxi_1-\bxi_2\|\\
    {}\qquad \quad\forall\,\bxi_1,\bxi_2\in \mathbb{R}^d,\,{\rm a.e.}\ \bx\in\Gamma_3. \end{array} \right\}
\label{eq:19a}
\end{equation}
Moreover, we assume
\begin{equation}
\kappa>0
\label{eq:20}
\end{equation}
on the microcrack diffusion coefficient, 
\begin{equation}
{\fb}_0\in C([0,T];L^2(\Omega)^d), \quad {\fb}_2\in C([0,T]; L^2(\Gamma_2)^d)
\label{eq:21}
\end{equation}
on the densities of forces and tractions,
\begin{equation}
\bu_0\in V,\quad \zeta_0\in K
\label{eq:22}
\end{equation}
on the initial data.  Here $K$ represents the set of admissible damage functions defined by
\begin{equation}
K=\left\{\xi \in Z\mid \xi\in [0,1]\ \mbox{a.e.\ in}\ \Omega \right\}.
\label{eq:23}
\end{equation}

By the Riesz representation theorem, we can define $\fb:[0,T]\to V$ by
\begin{equation}
(\fb(t), \bv)_V=\int_{\Omega}{\fb}_0(t)\cdot \bv\,dx+\int_{\Gamma_2} {\fb}_2(t)\cdot\bv\,da
\quad\forall\,\bv\in V,\ t\in[0,T].  \label{eq:24}
\end{equation}
Then conditions \eqref{eq:21} imply
\begin{equation}       
\fb \in C([0,T];V).  \label{eq:25}
\end{equation}
Let $a: Z\times Z\to\mathbb{R}$ be the bilinear form
\begin{equation}
a(\xi,\eta)=\kappa\,\int_\Omega\,\nabla\xi\cdot\nabla\eta\, dx,\quad \xi,\eta \in Z.   \label{eq:26}
\end{equation}

Let us introduce a weak formulation of the problem \eqref{eq:11}--\eqref{eq:18}. 

\begin{problem} \label{P1}
Find a displacement field $\bu:[0,T]\to V$, a stress field $\bsigma\colon [0,T]\to Q$, and a damage field 
$\zeta\colon [0,T]\to Z$ such that for all $t\in[0,T]$,
\begin{align}
&\bsigma(t)  = {\cal A}\bvarepsilon(\dot{\bu}(t))+{\cal B}(\bvarepsilon(\bu(t)),\zeta(t)),  \label{eq:27}\\[1mm]
&(\bsigma(t),\bvarepsilon(\bv-\dot {\bu}(t)))_Q+\int_{\Gamma_3} j^0(\dot\bu_\tau(t);\bv_\tau-\dot{\bu}_\tau(t))\,da
\geq (\fb(t), \bv-\dot {\bu}(t))_V\quad \forall\,\bv\in V,   \label{eq:28}
\end{align}
\begin{align}
\zeta(t)\in K,\quad &(\dot\zeta(t), \xi -\zeta(t))_{Z_0} +a(\zeta(t),\xi-\zeta(t)) \nonumber\\
& \ge(\phi(\bvarepsilon(\bu(t)),\zeta(t)),\xi-\zeta(t))_{Z_0} \quad\forall\, \xi \in K,  \label{eq:29}
\end{align}
and
\begin{equation}
\bu(0)=\bu_0,\quad \zeta(0)=\zeta_0.   \label{eq:30}
\end{equation}
\end{problem}

Denote by $c_\tau$ the smallest constant in the trace inequality
\begin{equation}
\|\bv_\tau\|_{L^2(\Gamma_3)^d} \le c_\tau \|\bv\|_V\quad\forall\,\bv\in V.
\label{eq:31}
\end{equation}

Similar to \cite[Theorem 5.1]{GOS15}, we can prove the following result.

\begin{theorem}\label{thm:exi}
Assume \eqref{eq:2}--\eqref{eq:22},
\begin{equation}
c_{2\tau}c_\tau^2<m_{\cal A}, 
\label{eq:31a}
\end{equation}
and either $\sqrt{2}\,c_{1\tau} c_\tau^2<m_{\cal A}$ or for a constant $d_\tau\ge 0$, 
$j^0(\bx,\bxi;-\bxi)\le d_\tau\left(1+\|\bxi\|\right)$ for all $\bxi\in\mathbb{R}^d$
and a.e.\ $\bx\in\Omega$.  Then Problem \ref{P1} has a unique solution $\bu\in C^1([0,T];V)$, 
$\bsigma\in C([0,T];Q)$ and $\zeta\in H^1(0,T;L^2(\Omega))\cap L^2(0,T;H^1(\Omega))$.  Moreover,
${\rm Div}\,\bsigma\in L^2(\Omega)^d$.
\end{theorem}

In the numerical solution of Problem \ref{P1}, it will be convenient to introduce the velocity variable
\begin{equation}
\bw(t)=\dot{\bu}(t).
\label{eq:32}
\end{equation}
Given the velocity $\bw(t)$ and the initial displacement $\bu(0)=\bu_0$ from \eqref{eq:30}, we can recover
the displacement by the formula
\begin{equation}
\bu(t)=\bu_0+\int_0^t \bw(s)\,ds.
\label{eq:33}
\end{equation}
Then \eqref{eq:27}--\eqref{eq:28} can be rewritten as 
\begin{align}
&\bsigma(t)  = {\cal A}\bvarepsilon(\bw(t))+{\cal B}(\bvarepsilon(\bu(t)),\zeta(t)),  \label{eq:34}\\[1mm]
&(\bsigma(t), \bvarepsilon(\bv -\bw(t)))_Q+ \int_{\Gamma_3} j^0(\bw_\tau(t);\bv_\tau -\bw_\tau(t))\,da
\geq (\fb(t), \bv-\bw(t))_V\quad \forall\,\bv\in V,   \label{eq:35}
\end{align}
for all $t\in[0,T]$.

\section{Numerical analysis of the weak formulation}\label{sec:NA}

In this section, we introduce and study a fully discrete numerical scheme to solve Problem \ref{P1}.
We assume the conditions stated in Theorem \ref{thm:exi} are valid so that Problem \ref{P1} has a unique solution.

For the approximation of the time derivative of the damage function, we use finite difference.
We divide the time interval $[0,T]$ uniformly and comment that much of 
the discussion of the numerical method below can be extended straightforward to the case of 
general partition of the time interval.  Thus, let $N$ be a positive integer, and define $k=T/N$ the step-size.
Then $0=t_0<t_1<\cdots<t_N=T$ is a uniform partition of $[0,T]$ with the nodes $t_n=nk$, $n=0,1,\cdots,N$.
For a function $z(t)$ continuous on $[0,T]$, we write $z_n=z(t_n)$.  We use the backward difference approximation
\begin{equation}
\dot\zeta(t_n)\approx \delta \zeta_n:=\frac{\zeta_n-\zeta_{n-1}}{k},\quad 1\le n\le N.
\label{NA1}
\end{equation}
For the spatial discretization, we use the finite element method.  For simplicity, we assume $\Omega$ is a
polygonal/polyhedral domain, and express the three parts of the boundary, $\Gamma_k$, $1\le k\le 3$, 
as unions of closed flat components with disjoint interiors:
\[ \overline{\Gamma_k}=\cup_{i=1}^{i_k}\Gamma_{k,i},\quad 1\le k\le 3.\]
Let $\{{\cal T}^h\}_h$ be a regular family of finite element partitions of $\overline{\Omega}$ into
triangular/tetrahedral elements, compatible with the partition of the boundary $\partial\Omega$ into 
$\Gamma_{k,i}$, $1\le i\le i_k$, $1\le k\le 3$, in the sense that if the intersection of one side/face 
of an element with one set $\Gamma_{k,i}$ has a positive measure with respect to $\Gamma_{k,i}$, 
then the side/face lies entirely in $\Gamma_{k,i}$.  Here $h\to0$ denotes the finite element mesh-size.
Corresponding to the partition ${\cal T}^h$, we introduce the linear finite element space
\begin{equation}
V^h = \left\{\bv^h\in C(\overline{\Omega})^d \mid \bv^h|_T\in \mathbb{P}_1(T)^d
   \ \ \forall\ T\in {\cal T}^h,\,\bv^h=\bzero\ {\rm on\ }\Gamma_1,\, v^h_\nu=0\ {\rm on\ }\Gamma_3\right\}
\label{NA2}
\end{equation}
for the displacement field, the piecewise constant finite element space
\begin{equation}
Q^h=\{\btau^h\in Q\mid \btau^h|_T\in\mathbb{R}^{d\times d}\ \forall\,T\in {\cal T}^h\}
\label{NA3}
\end{equation}
for the stress field, and the linear finite element space
\begin{equation}
Z^h=\left\{\xi^h\in C(\overline{\Omega})\mid\xi^h|_T\in \mathbb{P}_1(T)\ \forall\,T\in {\cal T}^h\right\}
\label{NA4}
\end{equation}
for the damage field.  Define the constrained subset of $Z^h$:
\begin{equation}
K^h=\left\{\xi^h\in Z^h \mid \xi^h|_T\in [0,1]\ \forall\,T\in {\cal T}^h\right\}.
\label{NA5}
\end{equation}

Let $\bu_0^h\in V^h$ and $\zeta_0^h\in K^h$ be appropriate approximations of $\bu_0$ and $\zeta_0$ such that
\begin{equation}
\|\bu_0-\bu^{h}_0\|_V\le c\,h,\quad \|\zeta_0-\zeta^h_0\|_{Z_0}\le c\,h.
\label{NA5a}
\end{equation}
These conditions are valid if, e.g., $\bu_0\in H^2(\Omega)^d$, $\zeta_0\in H^1(\Omega)$, and we define 
$\bu_0^h\in V^h$ to be the interpolant or $H^1(\Omega)^d$- or $L^2(\Omega)^d$-projection of $\bu_0$ onto $V^h$, 
define $\zeta_0^h\in K^h$ to be the $L^2(\Omega)$-projection of $\zeta_0$ onto $K^h$.  The smoothness 
conditions $\bu_0\in H^2(\Omega)^d$ and $\zeta_0\in H^1(\Omega)$ will follow from the solution regularities 
\eqref{NA12} and \eqref{NA14} below.

The discrete velocity and displacement approximations are denoted by $\{\bw^{hk}_n\}_{n=1}^N\subset V^h$ and 
$\{\bu^{hk}_n\}_{n=0}^N\subset V^h$, whereas the discrete stress and damage function approximations are 
denoted by $\{\bsigma^{hk}_n\}_{n=1}^N\subset Q^h$ and $\{\zeta^{hk}_n\}_{n=0}^N\subset K^h$.  Let 
${\cal P}_{Q^h}\colon Q\to Q^h$ be the orthogonal projection from $Q$ to $Q^h$, defined by 
\begin{equation}
{\cal P}_{Q^h}\bsigma\in Q^h,\quad ({\cal P}_{Q^h}\bsigma,\btau^h)_Q=(\bsigma,\btau^h)_Q\quad\forall\,\bsigma \in Q, \btau^h\in Q^h.
\label{NA6}
\end{equation}
Then a fully discrete scheme for Problem \ref{P1} is the following.

\begin{problem}\label{P1hk}
Find a discrete displacement field $\bu^{hk}=\{\bu^{hk}_n\}_{n=0}^N\subset V^h$, a discrete stress
field $\bsigma^{hk}=\{\bsigma^{hk}_n\}_{n=1}^N\subset Q^h$, and a discrete damage field
$\zeta^{hk}=\{\zeta^{hk}_n\}_{n=0}^N\subset K^h$ such that for $n=1,2,\dots,N$,
\begin{align}
& \bsigma^{hk}_n={\cal P}_{Q^h}{\cal A}\bvarepsilon(\bw^{hk}_n)
 +{\cal P}_{Q^h}{\cal B}(\bvarepsilon(\bu^{hk}_{n-1}),\zeta^{hk}_{n-1}),   \label{NA7} \\[1mm]
& (\bsigma^{hk}_n,\bvarepsilon(\bv^h -\bw^{hk}_n))_Q
+ \int_{\Gamma_3} j^0(\bw^{hk}_{n\tau};\bv^h_\tau-\bw^{hk}_{n\tau})\,da
\ge(\fb_n, \bv^h-\bw^{hk}_n)_V \quad \forall\, \bv^h\in V^h,   \label{NA8} \\[1mm]
& (\delta \zeta^{hk}_n, \xi^h -\zeta^{hk}_n)_{Z_0} +a(\zeta^{hk}_n,\xi^h-\zeta^{hk}_n)
\ge(\phi(\bvarepsilon(\bu^{hk}_{n-1}),\zeta^{hk}_{n-1}),\xi^h-\zeta^{hk}_n)_{Z_0}
\quad\forall\, \xi^h\in K^h,   \label{NA9}
\end{align}
and 
\begin{equation}
\bu^{hk}_0=\bu_0^h,\quad \zeta^{hk}_0=\zeta_0^h. \label{NA10}
\end{equation}
Here $\{\bu^{hk}_0\}_{n=0}^N$ and $\{\bw^{hk}_n\}_{n=1}^N$ are related by the equalities
\begin{equation}
\bw^{hk}_n=\delta\bu^{hk}_n \quad \mbox{and} \quad \bu^{hk}_n=\bu_0^h+k \sum_{i=1}^n \bw^{hk}_i.
\label{NA11}
\end{equation}
\end{problem}

Note that for implementation, \eqref{NA7} and \eqref{NA8} are combined together to give
\begin{align}
& ({\cal A}\bvarepsilon(\bw^{hk}_n),\bvarepsilon(\bv^h -\bw^{hk}_n))_Q
+({\cal B}(\bvarepsilon(\bu^{hk}_{n-1}),\zeta^{hk}_{n-1}),\bvarepsilon(\bv^h -\bw^{hk}_n))_Q
+\int_{\Gamma_3} j^0(\bw^{hk}_{n\tau};\bv^h_\tau-\bw^{hk}_{n\tau})\,da\nonumber\\
&\qquad \ge(\fb_n, \bv^h-\bw^{hk}_n)_V \quad \forall\, \bv^h\in V^h. \label{NA11a}
\end{align}

The solution existence and uniqueness of Problem \ref{P1hk} can be proved by an induction argument.  
The focus of the rest of this section is to bound the numerical solution errors.  For this purpose,
we assume the following additional solution regularities:
\begin{align}
&\bu\in W^{2,1}(0,T;V)\cap C([0,T];H^2(\Omega)^d), \label{NA12}\\
&\bsigma\in C([0,T];H^1(\Omega)^{d\times d}),  \label{NA13}\\
&\zeta\in H^2(0,T;Z_0)\cap C^1([0,T];Z)\cap C([0,T];H^2(\Omega)).   \label{NA14}
\end{align}
Then following the argument in \cite[Section 8.1]{HS2002}, we can show that for all $t\in (0,T)$,
\begin{align}
{\rm Div}\,\bsigma + {\fb}_0& =\bzero\quad\mbox{a.e.\ in}\ \Omega, \label{NA15}\\
\bsigma\bnu& ={\fb}_2 \quad \mbox{a.e.\ on}\ \Gamma_2,  \label{NA16}\\
u_\nu&=0\quad \mbox{a.e.\ on}\ \Gamma_3.  \label{NA17}
\end{align}

We first show the uniform boundedness of the numerical solution.

\begin{lemma}\label{lem:bd}
There exists a constant $M>0$, independent of $h$ and $k$, such that 
\[ \max_{1\le n\le N}\|\bw^{hk}_n\|_V \le M. \]
\end{lemma}
\begin{proof}
Let us fix $n\in\{1,\dots,N\}$.
We take $\bv^h=\bzero\in V^h$ in \eqref{NA11a},
\[ ({\cal A}\bvarepsilon(\bw^{hk}_n),\bvarepsilon(\bw^{hk}_n))_Q\le 
-({\cal B}(\bvarepsilon(\bu^{hk}_{n-1}),\zeta^{hk}_{n-1}),\bvarepsilon(\bw^{hk}_n))_Q
+\int_{\Gamma_3} j^0(\bw^{hk}_{n\tau};-\bw^{hk}_{n\tau})\,da+(\fb_n,\bw^{hk}_n)_V. \]
By \eqref{eq:2}\,(b),
\[ m_{\cal A}\|\bw^{hk}_n\|_V^2\le ({\cal A}\bvarepsilon(\bw^{hk}_n)-{\cal A}(\bzero),\bvarepsilon(\bw^{hk}_n))_Q.\]
Combining these relations, we obtain
\begin{align}
m_{\cal A}\|\bw^{hk}_n\|_V^2 & \le -({\cal A}(\bzero),\bvarepsilon(\bw^{hk}_n))_Q
 -({\cal B}(\bvarepsilon(\bu^{hk}_{n-1}),\zeta^{hk}_{n-1}),\bvarepsilon(\bw^{hk}_n))_Q\nonumber\\
&\quad{} +\int_{\Gamma_3} j^0(\bw^{hk}_{n\tau};-\bw^{hk}_{n\tau})\,da+(\fb_n,\bw^{hk}_n)_V.  \label{NA17a}
\end{align}

Let us treat each of the terms on the right side of \eqref{NA17a}.  Let $\delta>0$ be a small constant 
to be chosen later.  We recall the modified Cauchy-Schwarz inequality: for any $a,b\in\mathbb{R}$,
\begin{equation}
a\,b\le \delta\,a^2+c\,b^2,\quad c=\frac{1}{4\,\delta}. 
\label{mCS}
\end{equation}
Then,
\begin{equation}
-({\cal A}(\bzero),\bvarepsilon(\bw^{hk}_n))_Q \le \delta\,\|\bw^{hk}_n\|_V^2+c\,\|{\cal A}(\bzero)\|_Q^2.\label{NA17b}
\end{equation}
Similarly, 
\[ -({\cal B}(\bvarepsilon(\bu^{hk}_{n-1}),\zeta^{hk}_{n-1}),\bvarepsilon(\bw^{hk}_n))_Q
\le \delta\,\|\bw^{hk}_n\|_V^2+c\,\|{\cal B}(\bvarepsilon(\bu^{hk}_{n-1}),\zeta^{hk}_{n-1})\|_Q^2.\]
By \eqref{eq:3}\,(a),
\begin{align*}
\|{\cal B}(\bvarepsilon(\bu^{hk}_{n-1}),\zeta^{hk}_{n-1})\|
& \le \|{\cal B}(\bvarepsilon(\bu^{hk}_{n-1}),\zeta^{hk}_{n-1})-{\cal B}(\bzero,0)\|+\|{\cal B}(\bzero,0)\|\\
& \le L_{\cal B}\left(\|\bvarepsilon(\bu^{hk}_{n-1})\|+|\zeta^{hk}_{n-1}|\right)+\|{\cal B}(\bzero,0)\|.
\end{align*}
Then, by noting that $|\zeta^{hk}_{n-1}|\le 1$, we have
\[ \|{\cal B}(\bvarepsilon(\bu^{hk}_{n-1}),\zeta^{hk}_{n-1})\|_Q\le c\left(\|\bu^{hk}_{n-1}\|_V+1\right). \]
Now
\[ \bu^{hk}_{n-1}=\bu^{hk}_0+k\sum_{i=1}^{n-1}\bw^{hk}_i, \]
and so
\begin{align*}
\|\bu^{hk}_{n-1}\|_V & \le \|\bu^{hk}_0\|_V+k\sum_{i=1}^{n-1}\|\bw^{hk}_i\|_V,\\
\|\bu^{hk}_{n-1}\|_V^2 & \le c+c\,k\sum_{i=1}^{n-1}\|\bw^{hk}_i\|_V^2.
\end{align*}
Hence,
\begin{equation}
-({\cal B}(\bvarepsilon(\bu^{hk}_{n-1}),\zeta^{hk}_{n-1}),\bvarepsilon(\bw^{hk}_n))_Q
\le \delta\,\|\bw^{hk}_n\|_V^2+c\,k\sum_{i=1}^{n-1}\|\bw^{hk}_i\|_V^2+c.   \label{NA17c}
\end{equation}
Write
\[ j^0(\bw^{hk}_{n\tau};-\bw^{hk}_{n\tau})=\left[j^0(\bw^{hk}_{n\tau};-\bw^{hk}_{n\tau})
+j^0(\bzero;\bw^{hk}_{n\tau})\right]-j^0(\bzero;\bw^{hk}_{n\tau}).\]
By \eqref{eq:19a}\,(d),
\[ j^0(\bw^{hk}_{n\tau};-\bw^{hk}_{n\tau})+j^0(\bzero;\bw^{hk}_{n\tau})\le c_{2\tau}\|\bw^{hk}_{n\tau}\|^2. \]
By \eqref{eq:19a}\,(c),
\[ -j^0(\bzero;\bw^{hk}_{n\tau})\le c_{0\tau}\|\bw^{hk}_{n\tau}\|\le \delta\,\|\bw^{hk}_n\|^2+c. \]
Thus,
\begin{equation}
\int_{\Gamma_3} j^0(\bw^{hk}_{n\tau};-\bw^{hk}_{n\tau})\,da
\le \left(c_{2\tau}c_\tau^2+\delta\right)\|\bw^{hk}\|^2_V+c.  \label{NA17d}
\end{equation}
Finally, 
\begin{equation}
(\fb_n,\bw^{hk}_n)_V\le \delta\,\|\bw^{hk}\|^2_V+c\,\|\fb_n\|_V^2.  \label{NA17e}
\end{equation}
Use \eqref{NA17b}--\eqref{NA17e} in \eqref{NA17a},
\[ \left(m_{\cal A}-c_{2\tau}c_\tau^2-4\,\delta\right)\|\bw^{hk}_n\|_V^2 
\le c\,k\sum_{i=1}^{n-1}\|\bw^{hk}_i\|_V^2+c, \]
where the constant $c$ depends on $\|{\cal A}(\bzero)\|_Q$, $\|{\cal B}(\bzero,0)\|_Q$, $\|\fb\|_{C([0,T];V)}$,
and an upper bound on $\|\bu^{hk}_0\|_V$.  By choosing $\delta=(m_{\cal A}-c_{2\tau}c_\tau^2)/8$, we have
\[ \|\bw^{hk}_n\|_V^2 \le c\,k\sum_{i=1}^{n-1}\|\bw^{hk}_i\|_V^2+c. \]
Applying the Gronwall inequality, we get
\[ \|\bw^{hk}_n\|_V^2 \le c, \quad 1\le n\le N. \]
The proof is completed.  \hfill
\end{proof}

We will make use of several relations presented in \cite{SHS2006}.  By \cite[(3.25)]{SHS2006}, 
\begin{equation}
\|\bu_n-\bu^{hk}_{n-1}\|_V^2\le c\left(h^2+k^2\right)+c\,k\sum_{i=1}^{n-1}\|\bw_i-\bw^{hk}_i\|_V^2.
\label{NA19a}
\end{equation}
By \cite[(3.27)]{SHS2006}, 
\begin{equation}
\|\bu_n-\bu^{hk}_n\|_V^2\le c\left(h^2+k^2\right)+c\,k\sum_{i=1}^{n}\|\bw_i-\bw^{hk}_i\|_V^2.
\label{NA19}
\end{equation}
By \cite[(3.59)]{SHS2006}, 
\begin{equation}
\|\zeta_n-\zeta^{hk}_n\|_{Z_0}^2+ k \sum_{i=1}^n |\zeta_i-\zeta^{hk}_i|_{Z}^2
\le c\left(h^2+k^2+k\sum_{i=1}^{n-1}\|\bu_i-\bu^{hk}_i\|_V^2\right).
\label{NA20}
\end{equation}
By \cite[(3.71)]{SHS2006}, 
\begin{align}
\|\bsigma_n-\bsigma^{hk}_n\|_Q^2 & \le c\left(\|\bw_n-\bw^{hk}_n\|_V^2+\|\zeta_n-\zeta^{hk}_{n-1}\|_{Z_0}^2\right)
\nonumber\\
&\quad{} +c\,k \sum_{i=1}^{n-1}\|\bw_i-\bw^{hk}_i\|_V^2 +c\left(h^2+k^2\right).
\label{NA21}
\end{align}
From the regularity assumption \eqref{NA14}, we see that
\begin{equation}
\|\zeta_n-\zeta_{n-1}\|_{Z_0}=O(k).
\label{NA21z}
\end{equation}

By \eqref{eq:2}\,(b),
\begin{equation}
m_{\cal A}\|\bw_n-\bw^{hk}_n\|_V^2\le \left({\cal A}\bvarepsilon(\bw_n)-
{\cal A}\bvarepsilon(\bw_n^{hk}), \bvarepsilon(\bw_n-\bw_n^{hk})\right)_Q.
\label{NA21a}
\end{equation}
For any $\bv^h_n\in V^h$, write
\begin{align*}
\left({\cal A}\bvarepsilon(\bw_n)-{\cal A}\bvarepsilon(\bw_n^{hk}),\bvarepsilon(\bw_n-\bw_n^{hk})\right)_Q
& = \left({\cal A}\bvarepsilon(\bw_n)-{\cal A}\bvarepsilon(\bw_n^{hk}),\bvarepsilon(\bw_n-\bv_n^h)\right)_Q\\
&\quad{}+\left({\cal A}\bvarepsilon(\bw_n)-{\cal A}\bvarepsilon(\bw_n^{hk}),\bvarepsilon(\bv^h_n-\bw_n^{hk})\right)_Q.
\end{align*}
From \eqref{eq:34} at $t=t_n$,
\[ {\cal A}\bvarepsilon(\bw_n)=\bsigma_n-{\cal B}(\bvarepsilon(\bu_n),\zeta_n).\]
From \eqref{NA7},
\[ {\cal P}_{Q^h}{\cal A}\bvarepsilon(\bw^{hk}_n)=
\bsigma^{hk}_n-{\cal P}_{Q^h}{\cal B}(\bvarepsilon(\bu^{hk}_{n-1}),\zeta^{hk}_{n-1}).\]
Thus,
\begin{align*}
{\cal P}_{Q^h}({\cal A}\bvarepsilon(\bw_n)-{\cal A}\bvarepsilon(\bw^{hk}_n))
& = (\bsigma_n-\bsigma^{hk}_n)-(I-{\cal P}_{Q^h})\bsigma_n \\
&{}\quad -{\cal P}_{Q^h}({\cal B}(\bvarepsilon(\bu_n),\zeta_n)-{\cal B}(\bvarepsilon(\bu_{n-1}^{hk}),\zeta_{n-1}^{hk})).
\end{align*}
Then, from \eqref{NA6},
\begin{align*}
& \left({\cal A}\bvarepsilon(\bw_n)-{\cal A}\bvarepsilon(\bw_n^{hk}), \bvarepsilon(\bv^h_n-\bw^{hk}_n)\right)_Q\\[1.5mm]
&{}\qquad =(\bsigma_n-\bsigma^{hk}_n,\bvarepsilon(\bv^h_n-\bw^{hk}_n))_Q
 -((I-{\cal P}_{Q^h})\bsigma_n,\bvarepsilon(\bv^h_n-\bw^{hk}_n))_Q \\
&{}\qquad \quad -({\cal B}(\bvarepsilon(\bu_n),\zeta_n)-
{\cal B}(\bvarepsilon(\bu_{n-1}),\zeta_{n-1}),\bvarepsilon(\bv^h_n-\bw^{hk}_n))_Q\\[1.5mm]
&{}\qquad \le (\bsigma_n-\bsigma^{hk}_n,\bvarepsilon(\bv^h_n-\bw^{hk}_n))_Q\\
&{}\qquad \quad +c\left(\|\bu_n-\bu_{n-1}^{hk}\|_V+\|\zeta_n-\zeta_{n-1}^{hk}\|_{Z_0}\right)\|\bv^h_n-\bw^{hk}_n\|_V.
\end{align*}
Apply the above relations in \eqref{NA21a} to obtain
\begin{align}
m_{\cal A}\|\bw_n-\bw^{hk}_n\|_V^2
& \le c\,\|\bw_n-\bw_n^{hk}\|_V\|\bw_n-\bv^h_n\|_V+(\bsigma_n-\bsigma^{hk}_n,\bvarepsilon(\bv^h_n-\bw^{hk}_n))_Q\nonumber\\
&{}\quad +c\,\left(\|\bu_n-\bu_{n-1}^{hk}\|_V+\|\zeta_n-\zeta_{n-1}^{hk}\|_{Z_0}\right)\|\bv^h_n-\bw^{hk}_n\|_V\nonumber\\
&{}\quad  +c\,h\,\|\bv^h_n-\bw^{hk}_n\|_V.
\label{NA21b}
\end{align}
By the triangle inequality of the norm,
\[ \|\bv^h_n-\bw^{hk}_n\|_V\le \|\bw_n-\bv^h_n\|_V+\|\bw_n-\bw^{hk}_n\|_V. \]
For any $\delta>0$, we apply \eqref{mCS} and derive from \eqref{NA21b} that
\begin{align}
\left(m_{\cal A}-2\,\delta\right) \|\bw_n-\bw^{hk}_n\|_V^2 
& \le c\,\|\bw_n-\bv^h_n\|_V^2+(\bsigma_n-\bsigma^{hk}_n,\bvarepsilon(\bv^h_n-\bw^{hk}_n))_Q\nonumber\\
&{}\quad +c\left(h^2+\|\bu_n-\bu_{n-1}^{hk}\|_V^2 +\|\zeta_n-\zeta_{n-1}^{hk}\|_{Z_0}^2\right). \label{NA21c}
\end{align}
Using \eqref{NA19a} to bound the term $\|\bu_n-\bu_{n-1}^{hk}\|_V^2$, we get from \eqref{NA21c} that
\begin{align}
\left(m_{\cal A}-2\,\delta\right) \|\bw_n-\bw^{hk}_n\|_V^2 
& \le c\,\|\bw_n-\bv^h_n\|_V^2+(\bsigma_n-\bsigma^{hk}_n,\bvarepsilon(\bv^h_n-\bw^{hk}_n))_Q\nonumber\\
&{}\quad +c\left(h^2+k^2+k\sum_{i=1}^{n-1}\|\bw_i-\bw^{hk}_i\|_V^2+\|\zeta_n-\zeta_{n-1}^{hk}\|_{Z_0}^2\right). 
\label{NA22}
\end{align}

We take $\bv=\bw^{hk}_n\in V^h$ in \eqref{eq:35} at $t=t_n$ to get
\begin{equation}
-(\bsigma_n, \bvarepsilon(\bw^{hk}_n-\bw_n))_Q
\le \int_{\Gamma_3} j^0(\bw_{n\tau};\bw^{hk}_{n\tau} -\bw_{n\tau})\,da- (\fb_n,\bw^{hk}_n-\bw_n)_V.
\label{NA23}
\end{equation}
From \eqref{NA8},
\begin{equation}
-(\bsigma^{hk}_n,\bvarepsilon(\bv^h -\bw^{hk}_n))_Q
\le \int_{\Gamma_3} j^0(\bw^{hk}_{n\tau};\bv^h_\tau-\bw^{hk}_{n\tau})\,da-(\fb_n, \bv^h-\bw^{hk}_n)_V
\label{NA24}
\end{equation}
Write
\[ (\bsigma_n-\bsigma^{hk}_n,\bvarepsilon(\bv^h_n-\bw^{hk}_n))_Q
=(\bsigma_n,\bvarepsilon(\bv^h_n-\bw_n))_Q-(\bsigma_n,\bvarepsilon(\bw^{hk}_n-\bw_n))_Q
-(\bsigma^{hk}_n,\bvarepsilon(\bv^h_n-\bw^{hk}_n))_Q.\]
By using \eqref{NA23} and \eqref{NA24}, we then have
\begin{align*}
(\bsigma_n-\bsigma^{hk}_n,\bvarepsilon(\bv^h_n-\bw^{hk}_n))_Q 
&\le (\bsigma_n,\bvarepsilon(\bv^h_n-\bw_n))_Q-(\fb_n, \bv^h_n-\bw_n)_V\\
&\quad {} +\int_{\Gamma_3}\left[ j^0(\bw_{n\tau};\bw^{hk}_{n\tau} -\bw_{n\tau})
+ j^0(\bw^{hk}_{n\tau};\bv^h_{n\tau}-\bw^{hk}_{n\tau})\right] da.
\end{align*}
By the sub-additivity of the generalized directional derivative,
\[ j^0(\bw^{hk}_{n\tau};\bv^h_{n\tau}-\bw^{hk}_{n\tau})\le j^0(\bw^{hk}_{n\tau};\bv^h_{n\tau}-\bw_{n\tau})
+j^0(\bw^{hk}_{n\tau};\bw_{n\tau}-\bw^{hk}_{n\tau}).\]
By \eqref{eq:19a}\,(d),
\[ j^0(\bw_{n\tau};\bw^{hk}_{n\tau} -\bw_{n\tau})+j^0(\bw^{hk}_{n\tau};\bw_{n\tau}-\bw^{hk}_{n\tau})
\le c_{2\tau} \|\bw_{n\tau}-\bw^{hk}_{n\tau}\|_{\mathbb{R}^d}^2. \]
Hence,
\begin{align*}
(\bsigma_n-\bsigma^{hk}_n,\bvarepsilon(\bv^h_n-\bw^{hk}_n))_Q 
&\le (\bsigma_n,\bvarepsilon(\bv^h_n-\bw_n))_Q-(\fb_n, \bv^h_n-\bw_n)_V\\
&\quad {} +\int_{\Gamma_3}j^0(\bw^{hk}_{n\tau};\bv^h_{n\tau}-\bw_{n\tau})\, da
+ c_{2\tau} c_\tau^2 \|\bw_{n}-\bw^{hk}_{n}\|_V^2.
\end{align*}
Use this inequality in \eqref{NA22},
\begin{align}
\left(m_{\cal A}-c_{2\tau} c_\tau^2-2\,\delta\right)& \|\bw_n-\bw^{hk}_n\|_V^2\nonumber\\
& \le c\,\|\bw_n-\bv^h_n\|_V^2+ (\bsigma_n,\bvarepsilon(\bv^h_n-\bw_n))_Q-(\fb_n, \bv^h_n-\bw_n)_V  \nonumber \\
&{}\quad +\int_{\Gamma_3}j^0(\bw^{hk}_{n\tau};\bv^h_{n\tau}-\bw_{n\tau})\,da\nonumber\\
&{}\quad +c\left(h^2+k^2+k\sum_{i=1}^{n-1}\|\bw_i-\bw^{hk}_i\|_V^2+\|\zeta_n-\zeta_{n-1}^{hk}\|_{Z_0}^2\right). 
\label{NA25}
\end{align}
Recall the smallness assumption \eqref{eq:31a}.  We can choose $\delta>0$ sufficiently small and derive from 
\eqref{NA25} that
\begin{align*}
\|\bw_n-\bw^{hk}_n\|_V^2 
& \le c\left[\|\bw_n-\bv^h_n\|_V^2+ (\bsigma_n,\bvarepsilon(\bv^h_n-\bw_n))_Q-(\fb_n, \bv^h_n-\bw_n)_V\right]  \\
&{}\quad +c \int_{\Gamma_3}j^0(\bw^{hk}_{n\tau};\bv^h_{n\tau}-\bw_{n\tau})\, da\\
&{}\quad +c\left(h^2+k^2+k\sum_{i=1}^{n-1}\|\bw_i-\bw^{hk}_i\|_V^2+\|\zeta_n-\zeta_{n-1}^{hk}\|_{Z_0}^2\right). 
\end{align*}
With the use of \eqref{NA21z}, we rewrite the above inequality as
\begin{align}
\|\bw_n-\bw^{hk}_n\|_V^2 
& \le c\left[\|\bw_n-\bv^h_n\|_V^2+ (\bsigma_n,\bvarepsilon(\bv^h_n-\bw_n))_Q-(\fb_n, \bv^h_n-\bw_n)_V\right]  \nonumber\\
&{}\quad +c \int_{\Gamma_3}j^0(\bw^{hk}_{n\tau};\bv^h_{n\tau}-\bw_{n\tau})\, da\nonumber\\
&{}\quad +c\left(h^2+k^2+k\sum_{i=1}^{n-1}\|\bw_i-\bw^{hk}_i\|_V^2+\|\zeta_{n-1}-\zeta_{n-1}^{hk}\|_{Z_0}^2\right). 
\label{NA26}
\end{align}
From \eqref{NA20}, 
\[ \|\zeta_{n-1}-\zeta^{hk}_{n-1}\|_{Z_0}^2
\le c\left(h^2+k^2+k\sum_{i=1}^{n-2}\|\bu_i-\bu^{hk}_i\|_V^2\right). \]
Use \eqref{NA19},
\[ k\sum_{i=1}^{n-2}\|\bu_i-\bu^{hk}_i\|_V^2\le c\left(h^2+k^2\right)+c\,k\sum_{i=1}^{n-2}\|\bw_i-\bw^{hk}_i\|_V^2. \]
Hence, from \eqref{NA26},
\begin{align*}
\|\bw_n-\bw^{hk}_n\|_V^2 
& \le c\left[\|\bw_n-\bv^h_n\|_V^2+ (\bsigma_n,\bvarepsilon(\bv^h_n-\bw_n))_Q-(\fb_n, \bv^h_n-\bw_n)_V\right]\\
&{}\quad +c \int_{\Gamma_3}j^0(\bw^{hk}_{n\tau};\bv^h_{n\tau}-\bw_{n\tau})\, da 
+c\left(h^2+k^2+k\sum_{i=1}^{n-1}\|\bw_i-\bw^{hk}_i\|_V^2\right). 
\end{align*}
By an application of the Gronwall inequality,
\begin{align}
\|\bw_n-\bw^{hk}_n\|_V^2 
& \le c\left[\|\bw_n-\bv^h_n\|_V^2+(\bsigma_n,\bvarepsilon(\bv^h_n-\bw_n))_Q-(\fb_n, \bv^h_n-\bw_n)_V\right] \nonumber\\
&{}\quad +c \int_{\Gamma_3}j^0(\bw^{hk}_{n\tau};\bv^h_{n\tau}-\bw_{n\tau})\, da+c\left(h^2+k^2\right). 
\label{NA27}
\end{align}

We multiply \eqref{NA15} by an arbitrary $\bv\in V$, integrate over $\Omega$ and perform an integration by parts,
\[ \int_\Gamma \bsigma\bnu\cdot\bv\,da-\int_\Omega \bsigma\cdot\bvarepsilon(\bv)\,dx+\int_\Omega\fb_0\cdot\bv\,dx=0. \]
Split the integral over $\Gamma$ to three sub-integrals: the sub-integral over $\Gamma_1$ is zero since $\bv=\bzero$
on $\Gamma_1$; for the sub-integral over $\Gamma_2$, we apply the relation \eqref{NA16}; for the sub-integral over 
$\Gamma_3$, we use the relation \eqref{NA17}.   As a result,
\[ \int_\Omega \bsigma\cdot\bvarepsilon(\bv)\,dx-\int_\Omega\fb_0\cdot\bv\,dx
-\int_{\Gamma_2}\fb_2\cdot\bv\,da=\int_{\Gamma_3} \bsigma_\tau\cdot\bv_\tau da. \]
Thus, \eqref{NA27} can be reduced to
\begin{align}
\|\bw_n-\bw^{hk}_n\|_V^2 & \le c\left(\|\bw_n-\bv^h_n\|_V^2+h^2+k^2\right) \nonumber\\
&\quad{} +c\int_{\Gamma_3}\left[\bsigma_{n\tau}\cdot(\bv^h_{n\tau}-\bw_{n\tau})
+j^0(\bw^{hk}_{n\tau};\bv^h_{n\tau}-\bw_{n\tau})\right] da.
\label{NA28}
\end{align}
Since $\bsigma\in C([0,T];H^1(\Omega)^{d\times d})$ by \eqref{NA13}, we have
\[ \bsigma_\tau\in C([0,T];L^2(\Gamma_3))^d. \]
Thus,
\[ \int_{\Gamma_3} \bsigma_{n\tau}\cdot(\bv^h_{n\tau}-\bw_{n\tau})\, da
\le c\,\|\bsigma_\tau\|_{C([0,T];L^2(\Gamma_3)^d)}\|\bv^h_{n\tau}-\bw_{n\tau}\|_{L^2(\Gamma_3)^d}. \]
By Lemma \ref{lem:bd}, $\|\bw^{hk}_{n}\|_V$ is uniformly bounded.  Then, from \eqref{eq:19a}\,(c),
\[ j^0(\bw^{hk}_{n\tau};\bv^h_{n\tau}-\bw_{n\tau})\le 
 c\left(1+\|\bw^{hk}_{n\tau}\|\right)\|\bv^h_{n\tau}-\bw_{n\tau}\|,\]
we have a constant $c$ depending on the upper bound $M$ from Lemma \ref{lem:bd} that
\[ \int_{\Gamma_3}j^0(\bw^{hk}_{n\tau};\bv^h_{n\tau}-\bw_{n\tau})\, da
\le c\,\|\bv^h_{n\tau}-\bw_{n\tau}\|_{L^2(\Gamma_3)^d}. \]
Therefore, from \eqref{NA28}, we can derive the inequality
\begin{equation}
\|\bw_n-\bw^{hk}_n\|_V^2 \le c\left(\|\bw_n-\bv^h_n\|_V^2+\|\bv^h_{n\tau}-\bw_{n\tau}\|_{L^2(\Gamma_3)^d}+h^2+k^2\right)
\quad\forall\,\bv^h\in V^h.
\label{NA29}
\end{equation}

Based on \eqref{NA29}, \eqref{NA19} and \eqref{NA20}, we have proved the following C\'{e}a's inequality 
for error estimation.

\begin{theorem}\label{thm:Cea}
Assume the conditions stated in Theorem \ref{thm:exi}.
Let $(\bu,\zeta)$ be the solution of Problem \ref{P1}, $\bw=\dot{\bu}$, and let $(\bu^{hk},\bw^{hk},\zeta^{hk})$
be defined by Problem \ref{P1hk}.  Then under the solution regularity assumptions \eqref{NA12}--\eqref{NA14}, 
and \eqref{NA5a} on the initial values for the discrete problem, we have 
\begin{align}
& \max_n \|\bw_n-\bw^{hk}_n\|_V+\max_n \|\bu_n-\bu^{hk}_n\|_V
+\max_n\|\zeta_n-\zeta^{hk}_n\|_{Z_0}+ \left(k \sum_{n=1}^N |\zeta_n-\zeta^{hk}_n|_{Z}^2\right)^{1/2}\nonumber\\
&\qquad\le c\left(h+k\right) + c\max_n\inf_{\boldsymbol{v}^h_n\in V^h}
\left(\|\bw_n-\bv^h_n\|_V+\|\bv^h_{n\tau}-\bw_{n\tau}\|_{L^2(\Gamma_3)^d}^{1/2}\right).
\label{NA29a}
\end{align}
\end{theorem}

We can apply the standard finite element approximation theory (cf.\ \cite{AH2003, BS2008, Ci1978}) to bound 
the error 
\[ \max_n\inf_{\boldsymbol{v}^h_n\in V^h}
\left(\|\bw_n-\bv^h_n\|_V+\|\bv^h_{n\tau}-\bw_{n\tau}\|_{L^2(\Gamma_3)^d}^{1/2}\right)\]
in \eqref{NA29a}, and derive the next result from Theorem \ref{thm:Cea}.

\begin{corollary} \label{cor:order}
Keep the assumptions stated in Theorem \ref{thm:Cea}.
Under the additional solution regularity assumptions
\begin{equation}
\bw\in C([0,T];H^2(\Omega)^d),\quad \bw_\tau|_{\Gamma_{3,i}}\in C([0,T];H^2(\Gamma_{3,i})^d),\ 1\le i\le i_3,
\label{NA30}
\end{equation}
we have the following error bound:
\begin{align}
& \max_n \|\bw_n-\bw^{hk}_n\|_V+\max_n \|\bu_n-\bu^{hk}_n\|_V
+\max_n\|\zeta_n-\zeta^{hk}_n\|_{Z_0}+ \left(k \sum_{n=1}^N |\zeta_n-\zeta^{hk}_n|_{Z}^2\right)^{1/2}\nonumber\\
&\qquad \le c\left(h+k\right).
\label{NA31}
\end{align}
\end{corollary}

\section{Numerical results}\label{sec:ex}

In this section, we report some computer simulation results.   
We employ the Kelvin-Voigt type short memory viscoelastic law for the isotropic body, modified to reflect the 
damage effect on elastic properties of the body.
The viscosity operator $\mathcal{A}$ and the elasticity operator $\mathcal{B}$ are defined by
\begin{align}
\begin{split}
&\mathcal{A}(\bm{\tau})=2\,\phi\,\bm{\tau}+\xi\,\mbox{tr}(\bm{\tau})I,\qquad \bm{\tau}\in\mathbb{S}^2,\label{data1}\\
&\mathcal{B}(\bm{\tau}, \zeta) = \zeta\,(2\,\mu\,\bm{\tau} + \lambda\,\mbox{tr}(\bm{\tau})I),
\quad \bm{\tau} \in \mathbb{S}^2,\, \zeta \in [0,1],
\end{split}
\end{align}
where $I$ is the identity matrix, $\mbox{tr}$ is the trace operator on a matrix, $\mu$ and $\lambda$ are the
Lam\'{e} coefficients, whereas $\phi$ and $\xi$ represent the viscosity coeficients,  $\mu,\lambda,\phi,\xi > 0$. 
In all our simulations we take the following data
\begin{align}
\begin{split}
 &T=1,\label{data2}\\
 &\phi = \xi = 2, \quad \mu = \lambda = 4,\\
 &\bm{u}_{0}(\bm{x}) = (0,0), \quad \bm{x} \in \Omega,\\
 &j(\bm{\xi}) = 20\, \|\bm{\xi}\|, \quad \bm{\xi} \in \mathbb{R}^2,\\
 &\varphi(\bm{\tau}, \zeta) = \left \{ \begin{array}{ll}
   2\frac{1-\zeta}{\zeta} - 20\|\bm{\tau}\|^2, \quad \zeta \in [0.2, 1], \\
   8 - 20\|\bm{\tau}\|^2, \quad \zeta \in [0, 0.2), \\
  \end{array} \right. \bm{\tau} \in \mathbb{S}^2,\\
 &\kappa = 0.5.
\end{split}
\end{align}

We first demonstrate the effect of different partitions of the boundary and applied forces on the deformation of 
the body. In all cases, we show the initial configuration along with the shape of the body and damage field at the 
final time $t=1$. Choose a rectangular-shaped domain $\Omega=(0,2)\times(0,1) \subset \mathbb{R}^2$.  
For spatial discretization, we use uniform triangular partitions of $\overline{\Omega}$ and the corresponding
linear finite element spaces $\{V^h\}_h$; here $h$ represents the mesh-size such that the unit length part
$\{0\}\times [0,1]$ of the boundary is divided into $1/h$ equal size sub-intervals. For the temporal discretization, 
we use the uniform partitions of the time interval $[0,1]$ with the time step size $k=1/N$ for a positive integer $N$.
The numerical solutions correspond to the time step size $k=1/32$ and the mesh-size $h=1/32$. 

{\bf Experiment 1}: We take the following data
\begin{align*}
&\Gamma_1=\{0\}\times[0,1],\\
&\Gamma_2=([0,2]\times\{1\})\cup(\{2\}\times[0,1]))\cup([0,2]\times\{0\}),\\
&\Gamma_3=\emptyset,\\
 &\bm{f}_0(\bm{x},t) = (0,-0.2), \quad \bm{x} \in \Omega,\ t \in [0,T],\\
 &\bm{f}_2(\bm{x},t) = (0,0), \quad \bm{x} \in \Gamma_2,\ t \in [0,T].
\end{align*}
In this experiment we push the body down using a force with density $\bm{f}_0$. In Figure \ref{figOne} we observe 
that the body is curved downward. As a result of twisting forces, the damage inside the body gradually increases 
when the spatial point moves closer to $\Gamma_1$.

{\bf Experiment 2}: We change the data to
\begin{align*}
&\Gamma_1=\{0\}\times[0,1],\\
&\Gamma_2=([0,2]\times\{1\})\cup(\{2\}\times[0,1]))\cup([1,2]\times\{0\}),\\
&\Gamma_3=[0,1)\times\{0\},\\
 &\bm{f}_0(\bm{x},t) = (0,-0.8), \quad \bm{x} \in \Omega,\ t \in [0,T],\\
 &\bm{f}_2(\bm{x},t) = (0,0), \quad \bm{x} \in \Gamma_2,\ t \in [0,T].
\end{align*}
We once again push the body down, but in this case we introduce a rigid obstacle in contact with part of the body.
In Figure \ref{figTwo} we see that severe damage occurs in an area near the point $(1,0)$, which is a
corner point of the rigid foundation. There is damage also in the upper part of the body, as a result of expansion 
of the material.

{\bf Experiment 3}: In the final experiment we take
\begin{align*}
&\Gamma_1=\{0\}\times[0,1],\\
&\Gamma_2=([0,2]\times\{1\})\cup(\{2\}\times[0,1])),\\
&\Gamma_3=[0,2]\times\{0\},\\
 &\bm{f}_0(\bm{x},t) = (0,0), \quad \bm{x} \in \Omega,\ t \in [0,T],\\
 &\bm{f}_2(\bm{x},t) = (-1,-1), \quad \bm{x} \in \Gamma_2,\ t \in [0,T].
\end{align*}
In this case the entire bottom part of the body is in contact with a rigid obstacle, and we push the body down and 
to the left using a force with density $\bm{f}_2$.  In Figure \ref{figThree} we see that as a result of the action
of the boundary force, an increased amount of damage is observed towards the upper and right side of the body. 
It is also interesting to examine the effect of the frictional force on the interface $\Gamma_3$ between the foundation
and the body; the frictional force prevents the body from moving further to the left.

We find that simulation results from these experiments agree with our physical tuition.

\begin{figure}[ht]
\centering
\begin{minipage}{.49\textwidth}
  \centering
    \includegraphics[width=1.0\linewidth]{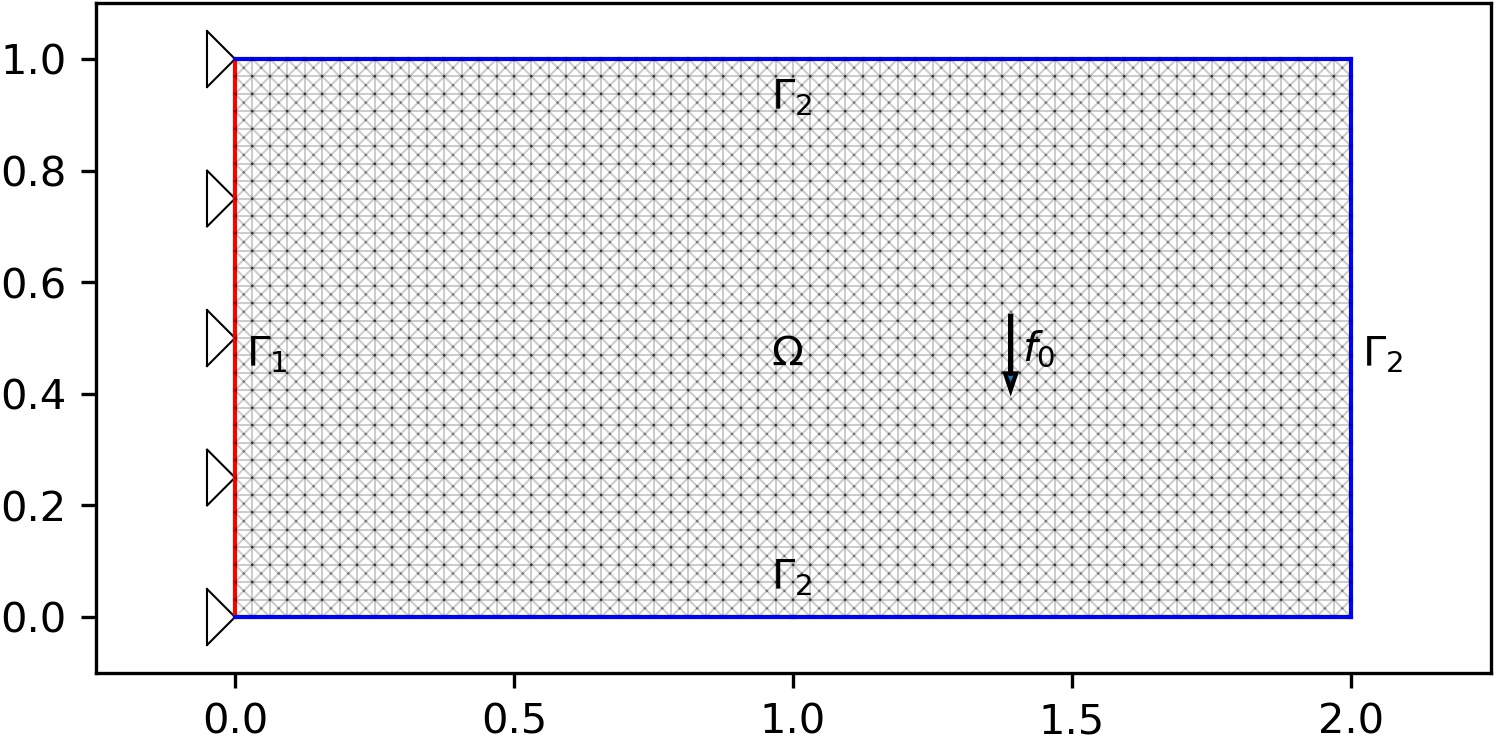}
\end{minipage}
\begin{minipage}{.49\textwidth}
  \centering
    \includegraphics[width=0.9\linewidth]{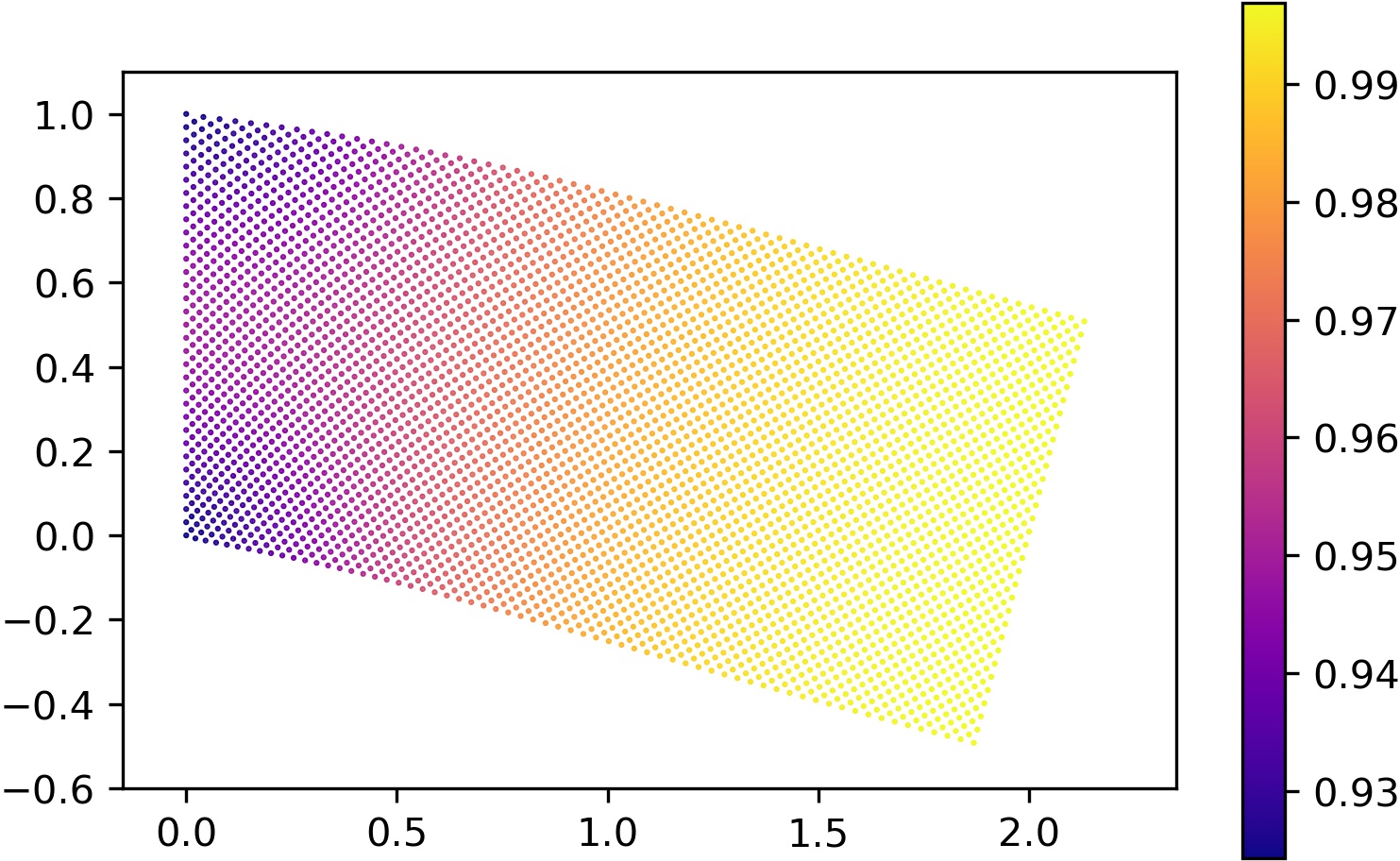}
\end{minipage}
 \caption{Initial configuration, body position and damage in first experiment} 
 \label{figOne}
\end{figure}

\begin{figure}[ht]
\centering
\begin{minipage}{.49\textwidth}
  \centering
    \includegraphics[width=1.0\linewidth]{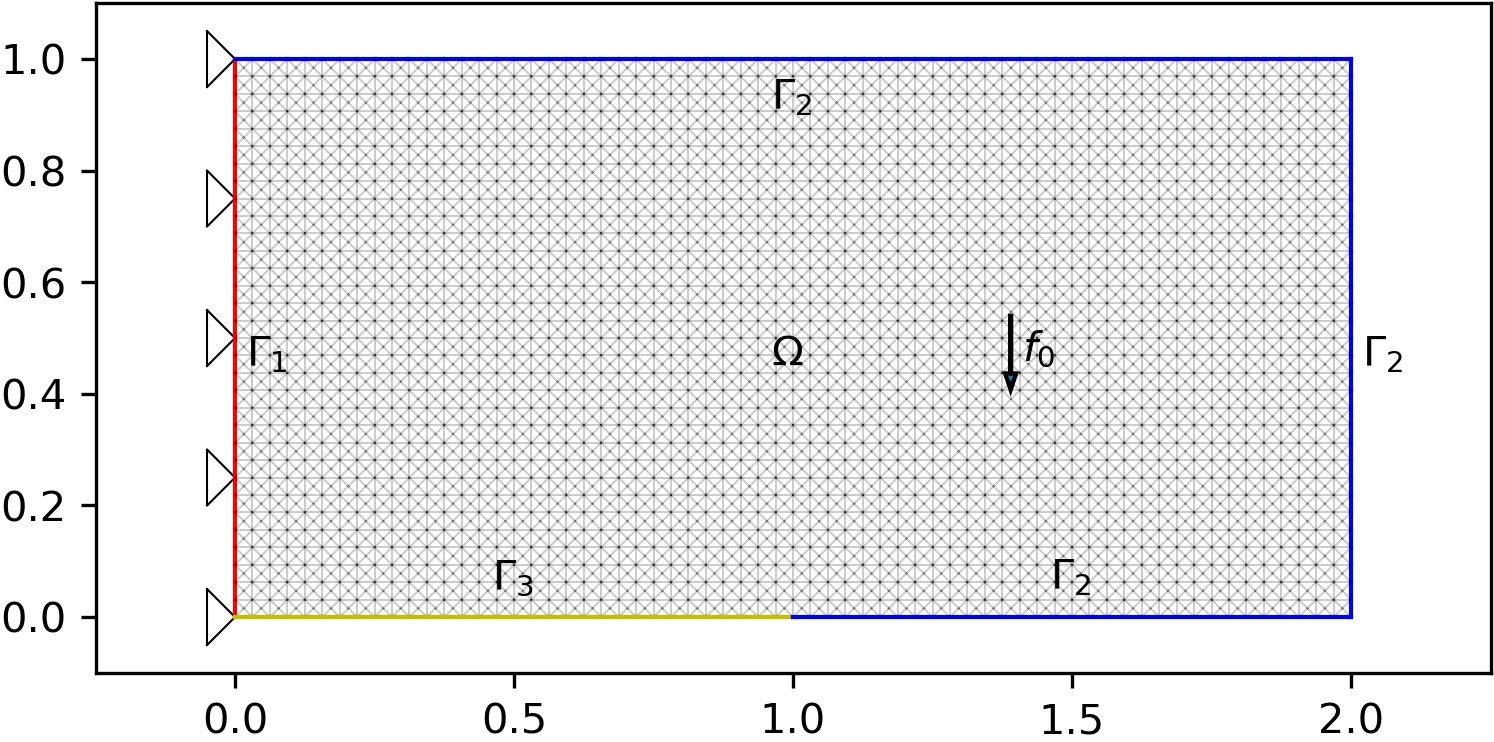}
\end{minipage}
\begin{minipage}{.49\textwidth}
  \centering
    \includegraphics[width=0.9\linewidth]{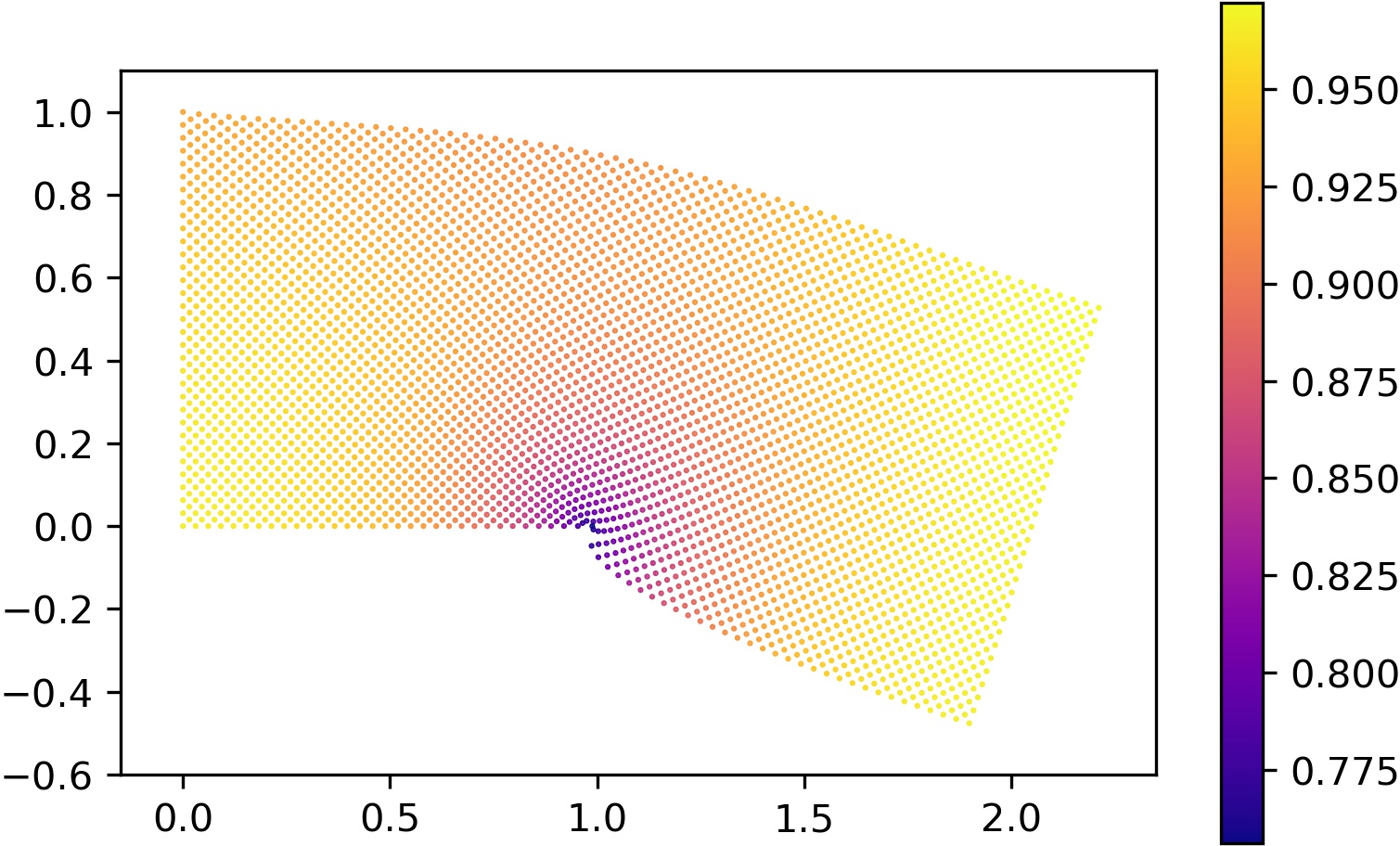}
\end{minipage}
 \caption{Initial configuration, body position and damage in second experiment} 
    \label{figTwo}
\end{figure}

\begin{figure}[ht]
\centering
\begin{minipage}{.49\textwidth}
  \centering
    \includegraphics[width=1.0\linewidth]{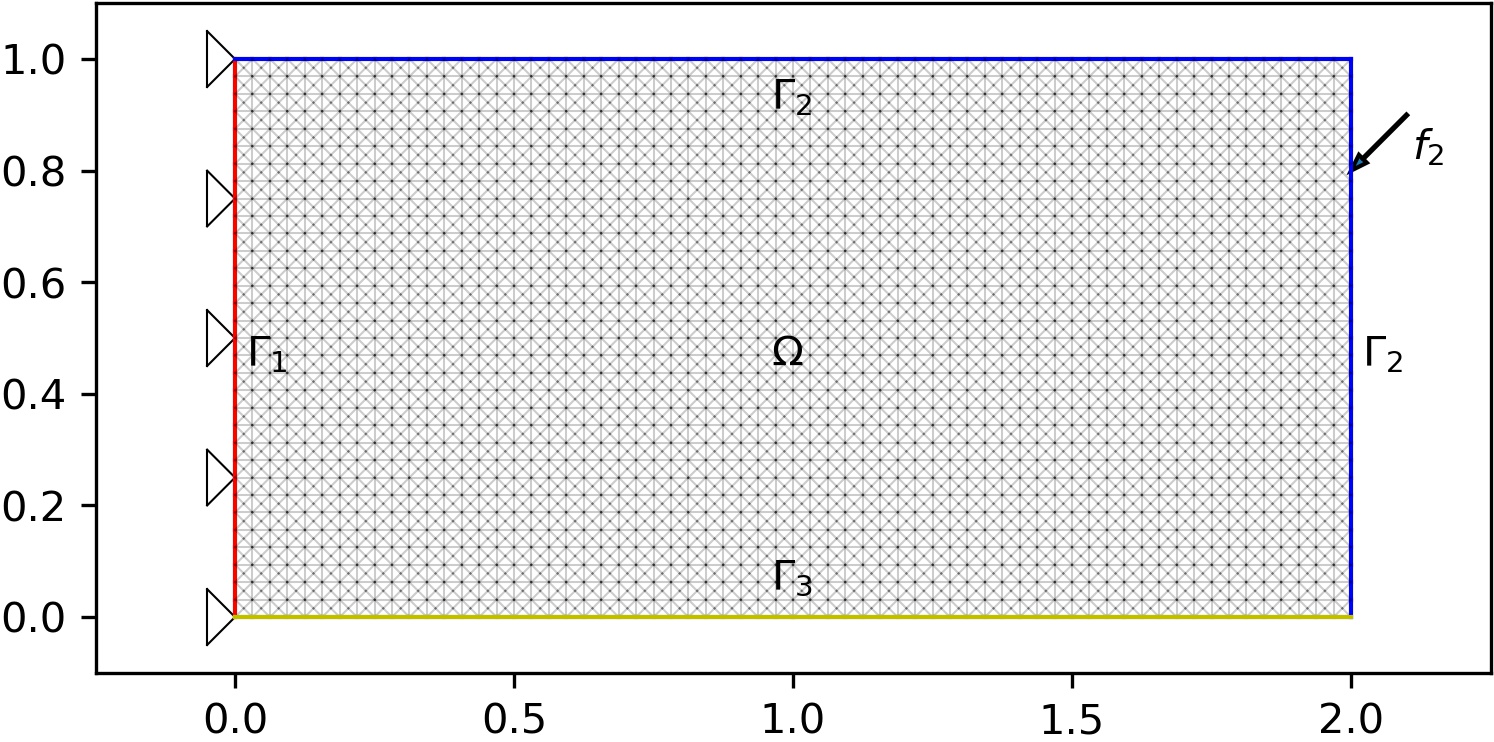}
\end{minipage}
\begin{minipage}{.49\textwidth}
  \centering
    \includegraphics[width=0.9\linewidth]{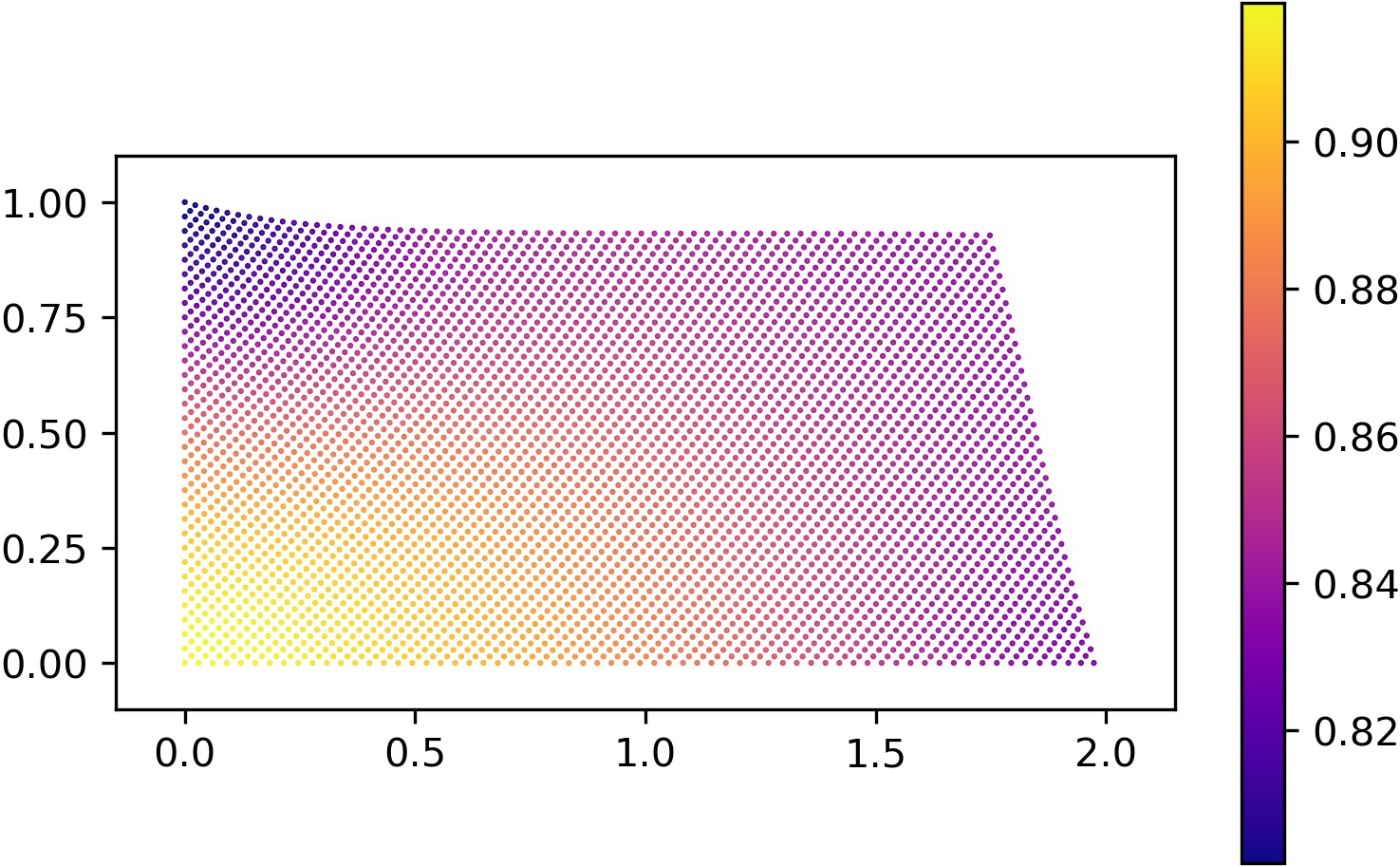}
\end{minipage}
 \caption{Initial configuration, body position and damage in third experiment} 
    \label{figThree}
\end{figure}

\begin{table}[ht]
\footnotesize
\centering
\begin{tabular}{ l r r r r r r }\hline
$h$ & $1/2$ & $1/4$ & $1/8$ & $1/16$ & $1/32$  \\ \hline
$\|\bm{w}-\bm{w}^{hk}\|_V/\|\bm{w}\|_V$ & $3.4889e-1$ & $2.2027e-1$ & $1.2745e-1$  &  $7.2811e-2$ & $3.9967e-2$ \\ 
{\rm Convergence order} & & 0.6635 &  0.7893 &  0.8077 &  0.8653 \\ \hline
$\|\zeta-\zeta^{hk}\|_V/\|\zeta\|_V$ & $8.7595e-2$ & $4.3672e-2$ & $1.7248e-2$ &  $6.1124e-3$ & $2.1316e-3$ \\ 
{\rm Convergence order} & & 1.0041 &  1.3402 &  1.4966 &  1.5198\\ \hline
\end{tabular}
\caption{Numerical errors vs.\ $h$ with fixed $k = 1/128$} \label{tabOne}
\end{table}

\begin{table}[ht]
\footnotesize
\centering
\begin{tabular}{ l r r r r r r }\hline
$k$ & $1/2$ & $1/4$ & $1/8$ & $1/16$ & $1/32$  \\ \hline
$\|\bm{w}-\bm{w}^{hk}\|_V/\|\bm{w}\|_V$ & $8.5393e-2$ & $6.8195e-2$ & $3.2105e-2$  &  $1.4883e-2$ & $7.0551e-3$ \\ 
{\rm Convergence order} & & 0.3244 &  1.0868 &   1.1091 &   1.0766 \\ \hline
$\|\zeta-\zeta^{hk}\|_V/\|\zeta\|_V$ & $4.0107e-1$ & $5.6124e-2$ & $2.2501e-2$  &  $8.4583e-3$ & $4.3778e-3$ \\ 
{\rm Convergence order} & & 2.8372 &  1.3186 &   1.4115 &   0.9502 \\ \hline
\end{tabular}
\caption{Numerical errors vs.\ $k$ with fixed $h = 1/128$} \label{tabTwo}
\end{table}

\begin{figure}[ht]
\centering
\begin{minipage}{.45\textwidth}
  \centering
  \includegraphics[width=0.9\linewidth]{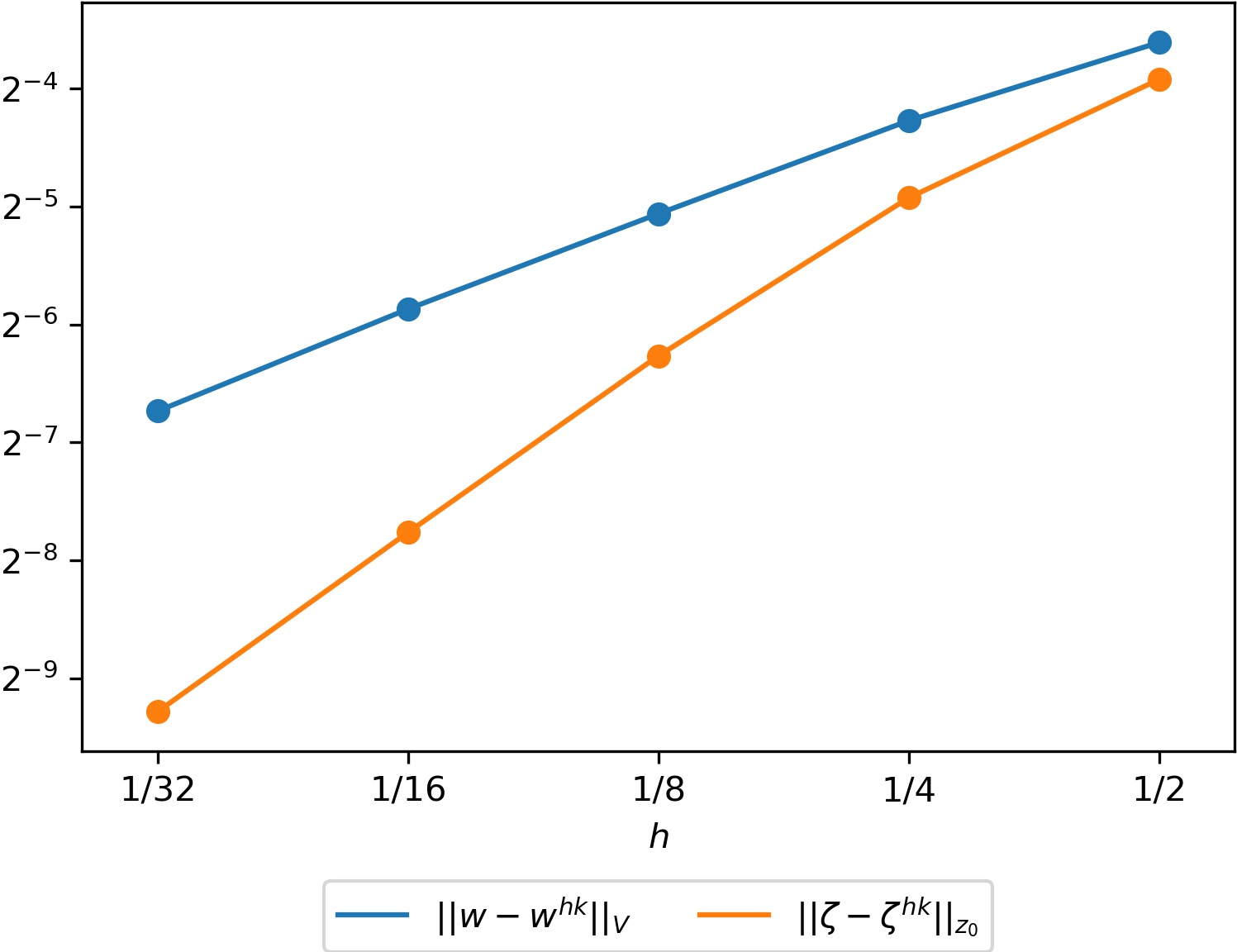}
  \caption{Errors vs.\ $h$ with fixed $k=1/128$} \label{figFour}
\end{minipage}
\begin{minipage}{.45\textwidth}
  \centering
  \includegraphics[width=0.9\linewidth]{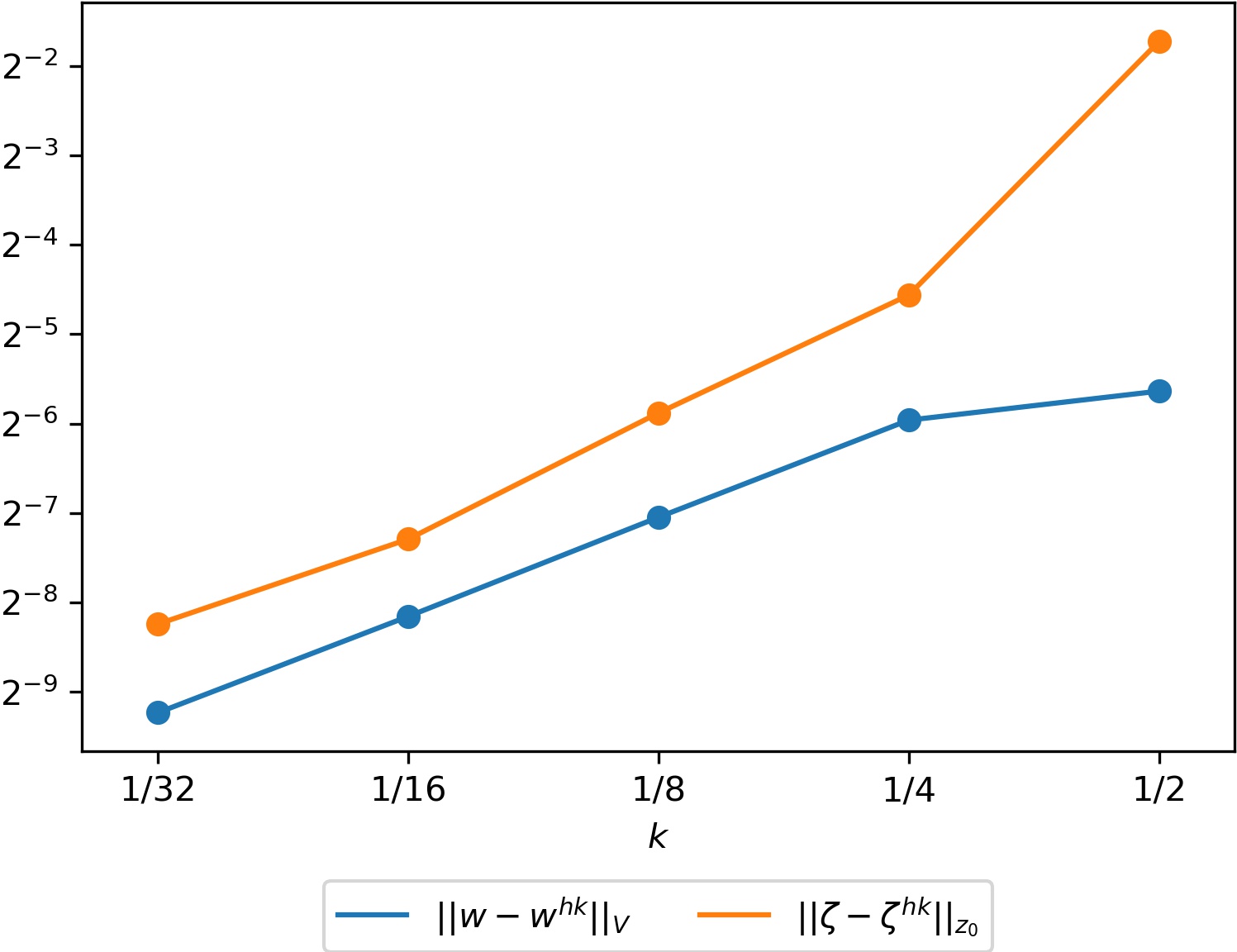}
  \caption{Errors vs.\ $k$ with fixed $h=1/128$} \label{figFive}
\end{minipage}
\end{figure}

We now turn to explore the convergence orders of the numerical solutions of a model problem. We take a square-shaped domain
$\Omega=(0,1)\times(0,1)$, use the same data as in (\ref{data1}) and (\ref{data2}) and
\begin{align*}
&\Gamma_1=\{0\}\times[0,1],\\
&\Gamma_2=([0,1]\times\{1\})\cup(\{1\}\times[0,1])),\\
&\Gamma_3=[0,1]\times\{0\},\\
 &\bm{f}_0(\bm{x},t) = (0,0), \quad \bm{x} \in \Omega,\ t \in [0,T],\\
 &\bm{f}_2(\bm{x},t) = (-1.4,-0.2), \quad \bm{x} \in \Gamma_2,\ t \in [0,T].
\end{align*}
We use uniform triangulations of the spatial domain and uniform partitions of the time interval, and 
let $h$ and $k$ be the spatial mesh-size and time step-size as defined above. We present a comparison 
of numerical errors $\|\bm{w} - \bm{w}^{hk}\|_V$ and  $\|\zeta - \zeta^{hk}\|_{Z_0}$  computed for a sequence 
of numerical solutions. The numerical solution corresponding to $h = 1/128$ and $k = 1/128$ is taken as the ``true''
solution $\bm{w}$ and $\zeta$ in computing the numerical errors; $\|\bm{w}\|_{V}\doteq 0.23525$ and
$\|\zeta\|_{Z_0}\doteq 0.75375$.

First, we fix $k = 1/128$ and start with $h = 1/2$, which is successively halved. The results are presented 
in Table \ref{tabOne} and Figure \ref{figFour}, where the dependence of the relative error estimates 
$\|\bm{w}  - \bm{w}^{hk}\|_V$  and $\|\zeta - \zeta^{hk}\|_{Z_0}$ with respect to $h$ are plotted on a log-log scale. 
Asymptotic convergence orders close to one for the velocity variable and slightly higher for the damage variable
can be observed for the numerical solutions.

Then, we fix $h = 1/128$ and start with $k = 1/2$, which is successively halved. The results are presented 
in Table \ref{tabTwo} and Figure \ref{figFive}. Asymptotic convergence orders close to one for both unknowns 
can be observed.

\medskip
\noindent {\bf Acknowledgments}.  
The project has received funding from the European Union's Horizon 2020 research 
and innovation programme under the Marie Sk{\l}odowska-Curie grant agreement No.\ 823731 CONMECH.

\end{document}